\documentclass[twocolumn, amsth]{autart}    % Enable this line and disable the 
\usepackage{graphicx}          % Include this line if your 
\usepackage{amsmath}
\usepackage{amssymb,mathtools}
\usepackage{subcaption}
\usepackage{wrapfig}
\usepackage{tikz}
\usepackage{tikz-cd}
\usepackage[bitstream-charter]{mathdesign}
\usepackage{hyperref}
\usepackage{stmaryrd}
\hypersetup{colorlinks=true,
citecolor=blue,linkcolor=blue
}
\usepackage{natbib}
\newtheorem{theorem}{Theorem}
\newtheorem{lemma}{Lemma}
\newtheorem{remark}{Remark}
\newtheorem{example}{Example}
\newtheorem{proposition}{Proposition}
\newtheorem{definition}{Definition}
\newtheorem{proof}{Proof}

\newcommand{\SO}{\textnormal{SO}}

\newcommand{\SE}{\textnormal{SE}}
\newcommand{\so}{\mathfrak{so}}

\newcommand{\se}{\mathfrak{se}}

\newcommand{\R}{\mathbb{R}}
\newcommand{\Sp}{\mathbb{S}}
\newcommand{\s}{\mathfrak{s}}
\newcommand{\g}{\mathfrak{g}}
\newcommand{\ad}{\normalfont{\text{ad}}}
\newcommand{\Hor}{\mathcal{H}}
\newcommand{\h}{\mathfrak{h}}
\newcommand{\D}{\mathcal{D}}
\newcommand{\frakd}{\mathfrak{d}}

\begin{document}

\begin{frontmatter}

\title{Nonholonomic mechanics and virtual constraints on Riemannian homogeneous spaces} % Title, preferably not more 

%\thanks[footnoteinfo]{This paper was not presented at any IFAC meeting. Corresponding author M.~T.~Cicero. Tel. +XXXIX-VI-mmmxxi. Fax +XXXIX-VI-mmmxxv.}

\author[ES]{Efstratios Stratoglou}, \ead{ef.stratoglou@alumnos.upm.es}
\author[AA]{Alexandre Anahory Simoes},  \ead{alexandre.anahory@ie.edu.es}    \author[AB]{Anthony Bloch}, \ead{abloch@umich.edu} 
\author[LC]{Leonardo J. Colombo} \ead{leonardo.colombo@car.upm-csic.es}

\address[ES]{Universidad Polit\'ecnica de Madrid (UPM), Jos\'e Guti\'errez Abascal, 2, 28006 Madrid, Spain.}
\address[AA]{IE School of Science and Technology, Paseo de la Castellana 259, Madrid - 28029 Madrid, Spain.}
\address[AB] {Department of Mathematics, University of Michigan, Ann Arbor, MI 48109, USA.} 
\address[LC]{Centre for Automation and Robotics (CSIC-UPM), Ctra. M300 Campo Real, Km 0,200, Arganda del Rey - 28500 Madrid, Spain.}
\thanks{The authors acknowledge financial support from Grant PID2022-137909NB-C21 funded by MCIN/AEI/ 10.13039/501100011033. A.B. was partially supported by NSF grant  DMS-2103026, and AFOSR grants FA
9550-22-1-0215 and FA 9550-23-1-0400.}
                             % keyword list or with the 
                                          % help of the Automatica 
                                          % keyword wizard

\begin{abstract}                          % Abstract of not more than 200 words.
Nonholonomic systems are, so to speak, mechanical systems with a prescribed restriction on the velocities. A virtual nonholonomic constraint is a controlled invariant distribution associated with an affine connection mechanical control system. A Riemannian homogeneous space is, a Riemannian manifold that looks the same everywhere, as you move through it by the action of a Lie group. These Riemannian manifolds are not necessarily Lie groups themselves, but nonetheless possess certain symmetries and invariances that allow for similar results to be obtained. In this work, we
introduce the notion of virtual constraint on Riemannian homogeneous spaces in a geometric framework which is a generalization of the classical controlled invariant distribution setting and we show the existence and uniqueness of a control law preserving the invariant distribution. Moreover we characterize the closed-loop dynamics obtained using the unique control law in terms of an affine connection. We illustrate the theory with new examples of nonholonomic control systems inspired by robotics applications.
\end{abstract}

\end{frontmatter}

\section{Introduction}
In Euclidean spaces, simple mechanical systems are governed by the equations of motion of the type
\begin{equation*}
    \ddot{q} = -\frac{\partial V}{\partial q} + F(q,\dot{q}),
\end{equation*}
where $q$ denotes the position of the system in the Euclidean space, $V$ is the potential function and $F$ denotes some external forces acting on the system. The equations of motion above, which can be identified with Newton's law of motion, have two major limitations. The first one is that as soon as we change coordinates, the form of the equations of motion necessarily changes if the transformation is not linear. The second is that if our system does not evolve in an Euclidean configuration space, such as a collection of joints in a manipulator, the equations above are usually not valid or, not globally valid in the entire configuration space. These facts explain why several authors have rewritten the equations of motion of mechanical systems in a coordinate-free setting that avoids both these drawbacks.  Riemannian geometry has already been implemented in nonlinear control theory to present a coordinate-free setting for simple mechanical systems (see for instance \cite{bloch2003nonholonomic} and \cite{bullo2019geometric}).

Given a Riemannian manifold $Q$ equipped with a Riemannian metric $\mathcal{G}$, $\nabla$ the Levi-Civita connection associated to the Riemannian metric $\mathcal{G}$, a potential function $V:Q\rightarrow \R$ and a force map $Y:TQ\rightarrow TQ$, a simple mechanical system is described by the following coordinate-free equations of motion
\begin{equation}\nabla_{\dot{q}} \dot{q} = -\text{grad }V + Y(q, \dot{q}).\end{equation}
where \text{grad} denotes the gradient vector field and $Y(q, \dot{q})$ encodes the external force. The previous coordinate-free equations translate how Newton's second law should be adapted in the setting of a Riemannian manifold: the covariant acceleration should be equal to the vector field corresponding to the forces acting on the system.

Nonholonomic systems are, roughly speaking, mechanical systems with constraints on their velocity that are not derivable from position constraints \cite{ne_mark2004dynamics}. They arise, for instance, in systems that have rolling contact (e.g., the rolling of wheels without slipping) or certain kinds of sliding contact (such as the sliding of skates). There are multiple applications in the context of wheeled motion, mobile robotics and robotic manipulation. Nonholonomic systems on Riemannian manifolds were studied in \cite{bloch1995nonholonomic} (see also \cite{bloch2003nonholonomic}).

A Riemannian homogeneous space (see \cite{helgason1979differential} for instance) is a Riemannian manifold that looks the same everywhere, as you move through it by the action of a Lie group. These Riemannian manifolds are not necessarily Lie groups themselves, but nonetheless possess certain symmetries and invariances that allow for similar results to be obtained. That is, by taking adventage of the symmetries we can simplify the dynamics of the system. One of the objectives of this paper is to study nonholonomic systems in a coordinate-free setting on Riemannian homogeneous spaces. Nonholonomic systems on Riemannian homogeneous spaces have not been considered in the literature before, but nevertheless some examples have been considered in the context of geometric control such as a sphere rolling on another sphere in \cite{jurdjevic1997geometric} (see also \cite{bloch2003nonholonomic} Section $7.4$, \cite{rojo2010rolling}), but a detailed geometric description of the dynamics of such systems has not been analyzed.

Virtual constraints are relations among the links of the
mechanism that are dynamically imposed through feedback
control. Their function is to coordinate the evolution of the
various links throughout a single variable—which is another
way of saying that they reduce the degrees of freedom—with
the goal of achieving a closed-loop mechanism whose
dynamic behaviour is fully determined by the evolution of
simplest lower-dimension system (see \cite{canudas2004concept} and \cite{westervelt2018feedback} for instance for an overview on virtual constraints).  Virtual constraints extend the application of zero dynamics to feedback design (see e.g.,\cite{isidori1985nonlinear}). Virtual holonomic constraints have been studied over the past years in different contexts, such as motion planning and control \cite{freidovich2008periodic}, \cite{mohammadi2018dynamic}, and biped
locomotion to achieve a desired walking gait \cite{chevallereau2003rabbit}, \cite{westervelt2003hybrid}. Virtual holonomic constraints on Riemannian homogeneous spaces have been explored in \cite{stratoglou2024virtualc}.

Virtual nonholonomic constraints are a class of virtual constraints that depend on velocities rather than only on the configurations of the system. Virtual constraints were introduced in \cite{griffin2015nonholonomic} to design a velocity-based swing foot placement in bipedal robots.  The work \cite{moran2021energy} (see also \cite{moran2023gymnastics}) introduces an approach to defining rigorously virtual nonholonomic constraints, but it is not set in the most appropriate geometric setting to study this kind of constraint: that of tangent bundles. 

In particular, a virtual nonholonomic constraint is described by a non-integrable distribution on the configuration manifold of the system for which there is a feedback control making it invariant under the flow of the closed-loop system. In \cite{Simoes:linear:nonholonomic} we provided sufficient conditions for the existence and uniqueness of such a feedback law defining the virtual linear nonholonomic constraint and we also characterize the trajectories of the closed-loop system as solutions of a mechanical system associated with an induced constrained connection. Moreover, we were able to produce linear nonholonomic dynamics by imposing virtual nonholonomic constraints on a mechanical control system. We extended the results to affine and nonlinear nonholonomic constraints in \cite{stratoglou2023bvirtual} and \cite{stratoglou2023geometry}, respectively. We also studied the design of virtual nonholonomic constraints on Lie groups in \cite{stratoglou2023virtual}. This paper goes one step further and studies the design of virtual nonholonomic constraints on Riemannian homogeneous spaces. The second objective of this paper is to show the existence and uniqueness of a control law preserving the invariant distribution. Moreover, in this paper, we also characterize the closed-loop nonholonomic dynamics obtained using the unique control law in terms of an affine connection, and we illustrate the theory with new examples of nonholonomic control systems inspired by robotics applications, in particular, inspired by a manipulator arm equipped with a knife cutting a spherical surface.

In \cite{stratoglou2024virtualc}, we have studied virtual holonomic constraints on Riemannian homogeneous spaces. In this paper, we work on virtual nonholonomic constraints. The main challenge in considering this class of constraints is the geometry behind the problem. In particular, we need to develop a mathematical theory for nonholonomic systems on Riemannian homogeneous spaces that has not yet been considered in the literature, as far as we know. Moreover, we developed new examples of such systems that were never studied neither from the modeling point of view nor in the control literature. Finally, there are many technical differences worth mentioning between \cite{stratoglou2024virtualc} and the present work. To mention some, one is related to the geometry of the constraints, which is now determined by a distribution rather than by a submanifold of the configuration manifold. Another difference is related with the fact that we drop the $\mathfrak{h}$-connection appearing in \cite{stratoglou2024virtualc}, since Riemannian geodesics on the Lie group $G$ acting on the homogeneous manifold are tangent to the horizontal distribution and, thus, projection to the horizontal bundle is not needed. Another major difference with the other paper is that we examine the reconstruction procedure and we prove that the closed-loop system on the homogeneous manifold is equivalent to a closed-loop system in the Lie algebra associated with the Lie group acting on the homogeneous manifold.

The remainder of the paper is structured as follows. Section \ref{Sec: background} introduces nonholonomic systems in a coordinate-free setting, with special emphasis on nonholonomic. systems on Lie groups. We study the dynamics of mechanical systems on Riemannian homogeneous spaces in Section \ref{sec3}, in particular the class of nonholonomic systems on Riemannian homogeneous spaces. We define virtual nonholonomic constraints in Section \ref{sec4}, where we provide sufficient conditions for the existence and uniqueness of a control law defining a virtual nonholonomic constraint.  Moreover, we introduce a constrained connection to characterize the closed-loop dynamics as a solution of the mechanical system associated with such a constrained connection. We study two particular applications in Sections \ref{ex1} and \ref{ex2}.

%In their study the authors make no distinction between making a constraint invariant under the closed-loop system or being stabilized by it. In our work, we only consider a constraint to be a virtual constraint when it is invariant under the controlled motion.

%\subsection{\textcolor{red}{Problem statement}}

%Given a mechanical controlled system of the form
%$$\nabla_{\dot{q}(t)} \dot{q}(t) = - \text{grad}_H \ V (q(t)) + u_{a}(t)Y^{a}(q(t))$$
%we would like to find a control law enforcing a virtual constraints. Taking advantages of the system's symmetries, we will simplify the problem at hand by lifting the dynamics to the Lie algebra. 
%\begin{figure}[htb!]
 %   \centering
  %  \begin{tikzcd}
%\mathfrak{h}\subseteq \mathfrak{g} \arrow[d, "T_{e}\pi"'] \arrow[r, "T_{e}L_{g}"] & H_{\tilde{q}}G\subseteq T_{\tilde{q}}G \arrow[d, "T_{\tilde{q}}\pi"] \\
%T_{\pi(e)}H \arrow[r, "T_{\pi(e)}\Phi_{g}"']                                      & T_{q}H
%\end{tikzcd}
 %   \caption{Commutative diagram: trajectories in a homogeneous manifold can be lifted to the Lie algebra.}
  %  \label{fig:enter-label}
%\end{figure}

\section{Nonholonomic systems in a coordinate-free setting}\label{Sec: background}
\subsection{Background on Riemannian manifolds}

Let $Q$ be an $n$-dimensional differentiable manifold equipped with a Riemannian metric $\left< \cdot, \cdot \right>$, i.e., a positive-definite symmetric covariant 2-tensor field. That is, to each point $q\in Q$ we assign a positive-definite inner product $\left<\cdot, \cdot\right>_q:T_qQ\times T_qQ\to\mathbb{R}$, where $T_qQ$ is the \textit{tangent space} of $Q$ at $q\in Q$ and $\left<\cdot, \cdot\right>_q$ varies smoothly with respect to $q$. We denote by $\tau_{q}:T_{q}Q \rightarrow Q$ the smooth projection assigning to each tangent vector $v_{q}$ the point $q$ at which the vector is tangent. The length of a tangent vector is defined as
$\|v_q\|=\langle v_q,v_q\rangle^{1/2}$ with $v_q\in T_q Q$. We denote by $T^{*}_{q}Q$, the dual space of $T_{q} Q$ composed of the $\R$-valued linear maps $T_{q}Q\to\R$. For any $q \in Q$, the Riemannian metric induces an invertible map $\flat: T_q Q \to T_q^\ast Q$, called the \textit{flat map}, defined by $\flat(X)(Y) = \left<X, Y\right>$ for all $X, Y \in T_q Q$. The inverse map $\sharp: T_q^\ast Q \to T_q Q$, called the \textit{sharp map}, is similarly defined implicitly by the relation $\left<\sharp(\alpha), Y\right> = \alpha(Y)$ for all $\alpha \in T_q^\ast Q$. Let $C^{\infty}(Q)$ and $\Gamma(TQ)$ denote the spaces of smooth scalar fields and smooth vector fields on $Q$, respectively. The gradient of a function on a Riemannian manifold is given by $\text{grad } f(p) = \sharp(df(q))$ for all $q \in Q$. The push-forward of the function $f$ will be denoted by $f_*$. 

{A vector bundle of rank $k>0$ on the manifold $Q$ is a smooth assignment to each point $q\in Q$ of a vector space with dimension $k$. Relevant particular cases for us are tangent bundles, where to each point $q$ we assign the tangent space at $q$; and distributions, in which the vector space is a subspace of the tangent space. If $P$ is a vector bundle on $Q$, then $P$ might be written as a collection of vector spaces $P=\cup_{q\in Q} P_{q}$, where $P_{q}$ is the $k$-dimensional space assigned to the point $q$. In addition, there is a map called the bundle projection and denoted by $\pi:P\to Q$, defined by $\pi(p)=q$ if $p\in P_{q}$. Smooth sections of a vector bundle $P$ on $Q$ are smooth maps $X:Q\to P$ having the property that $X(q)\in P_{q}$, i.e., the vector $X(q)$ must belong to the vector space assigned to $q$ for all $q\in Q$. We denoted the collection of smooth section of a vector bundle $P$ by $\Gamma(P)$.} Vector fields are a special case of smooth sections of vector bundles. In particular, they are smooth maps of the form $X: Q \to TQ$ such that $\tau_{Q} \circ X = \text{id}_Q$, the identity function on $Q$. The set of vector fields on $Q$ will be denoted by $\mathfrak{X}(Q)$. For a vector field $X\in \mathfrak{X}(Q)$ we define its vertical lift to $TQ$ to be $\displaystyle{X_{v_{q}}^{V}=\left. \frac{d}{dt}\right|_{t=0} (v_{q} + t X(q))}$, where $v_q\in T_qQ$. Note that $X^V\in\mathfrak{X}(TQ)$.

An \textit{affine connection} on $Q$ is a map $\tilde{\nabla}: \mathfrak{X}(Q) \times \mathfrak{X}(Q) \to \mathfrak{X}(Q)$ which is $C^{\infty}(Q)$-linear in the first argument, $\R$-linear in the second argument, and satisfies the product rule $\tilde{\nabla}_X (fY) = X(f) Y + f \tilde{\nabla}_X Y$ for all $f \in C^{\infty}(Q), \ X \in \mathfrak{X}(Q), \ Y \in \mathfrak{X}(Q)$. The connection plays a role similar to that of the directional derivative in classical real analysis. The operator
$\tilde{\nabla}_{X}$ which assigns to every smooth section $Y$ the vector field $\tilde{\nabla}_{X}Y$ is called the \textit{covariant derivative} (of $Y$) \textit{with respect to $X$}.

%{Connections induces a number of important structures on $Q$, a particularly ubiquitous such structure is the \textit{curvature endomorphism}, which is a map $R: \Gamma(TQ) \times \Gamma(TQ) \times \Gamma(E) \to \Gamma(TQ)$ defined by $R(X,Y)Z := \nabla_{X}\nabla_{Y}Z-\nabla_{Y}\nabla_{X}Z-\nabla_{[X,Y]}Z$ for all $X, Y \in \Gamma(TQ), \ Z \in \Gamma(E)$.  The curvature endomorphism measures the extent to which covariant derivatives commute with one another.} %We further define the \textit{curvature tensor} $\text{Rm}$ on $Q$ via $\text{Rm}(X, Y, Z, W) := \left<R(X, Y)Z, W\right>$.

Let $q: I \to Q$ be a smooth curve parameterized by $t \in I \subset \R$, and denote the set of smooth vector fields along $q$ by $\Gamma(q)$. Then for any affine connection $\tilde{\nabla}$ on $Q$, there exists a unique operator $\tilde{\nabla}_{\dot{q}}: \Gamma(q) \to \Gamma(q)$ (called the \textit{covariant derivative along $q$}) which agrees with the covariant derivative $\tilde{\nabla}_{\dot{q}}\tilde{W}$ for any extension $\tilde{W}$ of $W$ to $Q$. A vector field $X \in \Gamma(q)$ is said to be \textit{parallel along $q$} if $\displaystyle{\tilde{\nabla}_{\dot{q}} X\equiv 0}$. %For $k\in \N$, the $k$th-order covariant derivative of $W$ along $q$, denoted by $\displaystyle{D_t^k W}$, can then be inductively defined by $\displaystyle{D_t^k W = D_t \left(D_t^{k-1} W\right)}$. 
The covariant derivative allows to define a particularly important family of smooth curves on $Q$ called \textit{geodesics}, which are defined as the smooth curves $q$ satisfying $\tilde{\nabla}_{\dot{q}} \dot{q} = 0$. %{Moreover, geodesics induce a map $\mathrm{exp}_q:T_qQ\to Q$ called the \textit{exponential map} defined by $\mathrm{exp}_q(v) = \gamma(1)$, where $\gamma$ is the unique geodesic verifying $\gamma(0) = q$ and $\dot{\gamma}(0) = v$. In particular, $\mathrm{exp}_q$ is a diffeomorphism from some star-shaped neighborhood of $0 \in T_q Q$ to a convex open neighborhood $\mathcal{B}$ (called a \textit{goedesically convex neighborhood}) of $q \in Q$.} 

The Riemannian metric induces a unique torsion-free and metric compatible connection called the \textit{Riemannian connection}, or the \textit{Levi-Civita connection} (see \cite{boothby1986introduction}). Along the rest of the paper, we will assume that $\tilde{\nabla}$ is the Riemannian connection. 

%\textcolor{red}{Define a distribution and the nonholonomic connection $\nabla^{nh}$. Express nonholonomic trajectories in terms of the nonholonomic connection.}

In the presence of a constraint distribution $\mathcal{D}$, we can define two
complementary orthogonal projectors ${\mathcal P}\colon TQ\to {\mathcal D}$ and ${\mathcal Q}\colon TQ\to {\mathcal
D}^{\perp},$ where $\mathcal{D}^\perp$ is the orthogonal complement of $\mathcal{D}$ i.e. for any $q\in Q$, $\mathcal{D}^\perp_g$ contains all vectors that are orthogonal to every vector in $\mathcal{D}_g$ with respect to the Riemannian metric. Thus, we can split the tangent bundle in the orthogonal decomposition $\mathcal{D}\oplus\mathcal{D}^{\perp}=TQ$. In what follows we will consider only non-integrable distributions.

Consider the \textit{nonholonomic connection} $\tilde{\nabla}^{nh}:\mathfrak{X}(Q)\times \mathfrak{X}(Q) \rightarrow \mathfrak{X}(Q)$ defined by (see \cite{bullo2019geometric} for instance)
\begin{equation}\label{nh:connection}
    \tilde{\nabla}^{nh}_X Y =\tilde{\nabla}_{X} Y + (\tilde{\nabla}_{X} \mathcal{Q})(Y).
\end{equation}

Next, consider  mechanical systems where the Lagrangian is of mechanical type, that is, mechanical systems with a dynamics described by a Lagrangian function $L:TQ\rightarrow\mathbb{R}$ defined by
\begin{equation}\label{mechanical:lagrangian}
    L(v_q)=\frac{1}{2}\left\langle v_q, v_q \right\rangle - V(q),
\end{equation}
with $v_q\in T_qQ$, where 
$V:Q\rightarrow\mathbb{R}$ is a (smooth) potential function, and also assume that the Lagrangian system is subject to the nonholonomic constraints given by the distribution $\mathcal{D}$. Then, the \textit{nonholonomic trajectories}, i.e., the trajectories of the nonholonomic mechanical system associated with the Lagrangian \eqref{mechanical:lagrangian} and the distribution $\mathcal{D}$ must satisfy the following equation
\begin{equation}\label{nonholonomic:mechanical:equation}
    \tilde{\nabla}^{nh}_{\dot{q}}\dot{q} + \mathcal{P}(\text{grad } V(q(t))) = 0.
\end{equation}

Here, the vector field $\text{grad }V\in\mathfrak{X}(Q)$ is characterized by $$\left\langle\text{grad } V, X \right\rangle = dV(X), \; \mbox{ for  every } X \in
\mathfrak{X}(Q).$$

Of course, an important particular case is when there are no constraints, in which case $\D=TQ$, $\tilde{\nabla}^{nh}=\tilde{\nabla}$ and the equations above become the Euler-Lagrange equations for a system with Lagrangian given by \eqref{mechanical:lagrangian} in Riemannian form:
$$\tilde{\nabla}_{\dot{q}} \dot{q} + \text{grad } V (q(t))=0.$$

\subsection{Riemannian geometry and Lie groups}\label{sec: background_Lie}

Let $G$ be a Lie group and its Lie algebra $\g$ be defined as the tangent space to $G$ at the identity, $\mathfrak{g}:=T_{e}G$. Denote by $L_g$ the left-translation map $L_g:G\to G$ given by $L_g(h)=gh,$ for all $g,h\in G,$ where $gh$ denotes the Lie group multiplication between the elements $g$ and $h$.

%One of the most important aspects of group theory is its role in understanding the symmetry properties of a particular set via \textit{group actions}. That is, the types of groups that are able to act upon a set (together with the properties of each particular group action) can illuminate important properties related to the structure and symmetry of the set itself. The same can be said about Lie groups, where now the sets that they act upon are smooth manifolds. 

%\textcolor{magenta}{Tony: I don't see how the last sentence in the paragraph above fits with the rest. Omit?}

%\textcolor{red}{Leo: Done.}

\vspace{.2cm}

\begin{definition}\label{def: left-action}
Let $G$ be a Lie group and $H$ be a smooth manifold. A \textit{left-action} of $G$ on $H$ is a smooth map $\Phi: G \times H \to H$ such that
\begin{enumerate}
    \item $\Phi(e, q) = q$,
    \item $\Phi(g, \Phi(h, q)) = \Phi(gh, q)$,
    \item The map $\Phi_g: H \to H$ defined by $\Phi_g(q) = \Phi(g, q)$ is a diffeomorphism,
\end{enumerate}
for all $g,h \in G,$ $q \in H$.
\end{definition}

We now define an important class of group actions:

\vspace{.2cm}

\begin{definition}\label{def: types of group actions}
Let $G$ be a Lie group, $H$ be a smooth manifold, and $\Phi: G \times H \to H$ be a left-action. We say that $\Phi$ is \textit{transitive} if for every $p,q \in H$, there exists some $g \in G$ such that $gp = q$. \end{definition}

One of the most important cases of a Lie group action is a Lie group acting on itself. The left action $L_{g}$ is a left action of $G$ on itself and a diffeomorphism on $G$.  Its tangent map  (i.e, the linearization or tangent lift) is denoted by $T_{h}L_{g}:T_{h}G\to T_{gh}G$. %Similarly, the cotangent map (cotangent lift), is defined as $(T_{h}L_g)^{*}$, the dual map of the tangent lift denoted by $T_{h}^{*}L_{g}:T^{*}_{h}G\to T^{*}_{gh}G$, and determined by the relation $\langle(T_hL_g)^{*}(\alpha_{gh}), Y_h\rangle=\langle\alpha_{gh},(T_hL_g)Y_h\rangle$, $Y_h\in T_{h}G$, $\alpha_{gh}\in T^{*}_{gh}G$. It is well known that the tangent and cotangent lift are actions.
Let us denote the set of vector fields on a Lie group $G$ by $\mathfrak{X}(G)$. A \textit{left invariant} vector field is an element $X$ of $\mathfrak{X}(G)$ such that $T_{h}L_g(X(h))=X(L_g(h))=X(gh)$ $\forall\,g,h\in G$. We denote the vector space of left-invariant vector fields on $G$ by $\mathfrak{X}_L(G).$ Consider the isomorphism $(\cdot)_{L}:\g\to\mathfrak{X}_L(G)$ given by $\xi_{L}(g)=(T_{e}L_{g})(\xi),$ where $\xi\in\g$ and $g\in G.$ Note that under this isomorphism we have $\g\simeq\mathfrak{X}_L(G).$ 

%{Give the definition of the adjoint map....}

%We define an operator $\ad^\dagger: \g \times \g \to \g$ by $\ad^\dagger_{\xi} \eta = (\ad^\ast_{\xi} \eta^\flat)^\sharp$ for all $\xi, \eta \in \g$. It follows that $\ad^\ast$ is bilinear, and satisfies
%$\left<\ad^\dagger_\xi \eta, \sigma \right> = \left<\ad^\ast_\xi \eta^\flat, \sigma \right>_\ast = \left<\eta^\flat, \ad_\xi \sigma \right>_\ast = \left<\eta, \ad_\xi \sigma \right>$ for all $\xi, \eta, \sigma \in \g$. Hence, $\ad^\ast$ is nothing more than the adjoint of the adjoint action with respect to the Riemannian metric on $G$.

Consider an inner-product on the Lie algebra $\g$ denoted by $\langle \cdot,\cdot\rangle_\g$. Using the left-translation we define a Riemannian metric on $G$ by the relation $\langle X,Y\rangle:=\langle (T_{g} L_{g^{-1}}) (X),(T_{g} L_{g^{-1}}) (Y)\rangle_\g$ for all $g\in G, \; X,Y\in T_g G$, which is called left-invariant metric because $\langle (T_{g}L_{h})(X),(T_{g}L_{h})(X)\rangle=\langle X,Y\rangle$ for all $g, h\in G, \; X,Y\in T_g G$. Let $\tilde{\nabla}$ be the Levi-Civita connection on $G$ associated to the above metric. The isomorphism $(\cdot)_{L}:\g\to\mathfrak{X}_L(G)$ helps us define an operator $\tilde{\nabla}^\g:\g\times\g\to\g$ by 
\[\tilde{\nabla}_\xi^\g \eta:=\tilde{\nabla}_{\xi_{L}} \eta_{L}(e)\]
for all $\xi,\eta\in\g.$ Although $\tilde{\nabla}^\g$ is not a connection we will refer to it as the Riemannian $\g-$connection corresponding to $\tilde{\nabla}$ (see \cite{goodman2024reduction} for more details).

We denote by $\g^{*}$, the dual space of $\g$ composed by the $\R$-valued linear maps $\g\to\R$. The dual space $\g^{*}$ is a vector space isomorphic to $\g$ itself through the musical isomorphism $\flat_{\g}:\g\to \g^{*}$ defined by the formula $\flat_{\g}(\eta) (\xi) = \langle \eta, \xi \rangle_{\g}$ for $\eta, \xi\in\g$. In the next result, we will make use of the adjoint action $\ad_{\xi}:\g\to\g$ and its dual map $\ad_{\xi}^\ast$ with respect to the Riemannian metric on $G$. This leads to the following expression for $\tilde{\nabla}^\g$ (see Theorem $5.40$ of \cite{bullo2019geometric}, for instance)
\begin{lemma}\label{Riemanian:g:conn}
    The  Riemannian  $\g-$connection satisfies: \[\tilde{\nabla}^\g_\xi\eta=\frac{1}{2}\left([\xi,\eta]_\g-\sharp_{\g} \left[ \text{ad}^{*}_{\xi} \flat_{\g} (\eta)\right] - \sharp_{\g} \left[ \text{ad}^{*}_{\eta} \flat_{\g} (\xi)\right] \right) \]
    for all $\xi,\eta\in\g$, where $\sharp_{\g}:\g^{*}\to\g$ is the inverse map of $\flat_{\g}$.
\end{lemma}

\vspace{.2cm}

\begin{lemma}[see \cite{goodman2024reduction}]\label{lemma2}
Consider a Lie group $G$ with Lie algebra $\g$ and left-invariant Levi-Civita connection $\nabla$. Let $g: [a,b] \to G$ be a smooth curve and $X$ a smooth vector field along $g$. Then the following relation holds for all $t \in [a, b]$:
\begin{align}
    \tilde{\nabla}_{\dot{g}} X(t) = g(t) \left(\dot{\eta}(t) + \tilde{\nabla}_{\xi}^\g \eta(t) \right) ,\label{Cov-to-cov}
\end{align}
where $\xi(t) = g(t)^{-1}\dot{g}(t) $ and $\eta(t) = g(t)^{-1}X(t)$.
\end{lemma}

The geodesic equation on a Lie group equipped with a left-invariant metric might be recast as an equation on the Lie algebra, the Euler-Poincar\'e equations, as the well-known result below establishes.

%From the previous Lemma, if $g(t)$ is a geodesic with respect to the Levi-Civita connection, then $g \left( \dot{\xi} + \nabla^\g_{\xi} \xi\right) = 0$, where $\xi := g^{-1} \dot{g} $, and we obtain the Euler-Poincare equations for geodesics:
\begin{theorem}\label{thm: EP_geo}
Suppose that $g: [a, b] \to G$ is a geodesic with respect to the Levi-Civita connection, and let $\xi :=g^{-1} \dot{g} $. Then, $\xi$ satisfies on $[a, b]$ the Euler-Poincar\'e equations
\begin{align}
    \dot{\xi} + \tilde{\nabla}^\g_{\xi} \xi = 0.\label{EP: geo}
\end{align}
%or equivalently
%\begin{equation}
%    \dot{\xi}(t) + \sharp\left[ \text{ad}^{*}_{\xi(t)}\flat(\xi(t))\right] = 0.
%\end{equation}
\end{theorem}

\subsection{Nonholonomic systems on Lie Groups}

Consider a left-invariant distribution $\mathcal{D}$ on the Lie Group $G$, that is, for each $g\in G$, the fiber at $g$, denoted by $\D_{g}$, is defined by $T_{e}L_{g}(\frakd)$, 
where $\frakd$ is a subspace of the Lie algebra $\g$. Using the inner product $\langle \cdot,\cdot \rangle_{\g}$ on the Lie algebra, we may define the orthogonal subset to $\frakd$ by $\frakd^{\bot}=\{\xi\in\g \; : \; \langle \xi,\eta \rangle_{\g}=0,\; \forall \eta\in\frakd \}$, then $\g=\frakd\oplus\frakd^\bot.$ Finally, consider the orthogonal projectors $\mathfrak{P}: \g \to \frakd$  and $\mathfrak{Q}:\g \to \frakd^{\bot}$.

Given a left invariant metric on $G$ and a left-invariant distribution $\D$, consider the associated orthogonal distribution $\D^{\bot}$ and the orthogonal projectors $\mathcal{P}: TG \to \D$ and $\mathcal{Q}:TG\to \D^{\bot}$. In  \cite{stratoglou2023virtual}, it has been shown that the orthogonal distribution $\D^{\bot}$ is left-invariant and $\D^{\bot}_{g} = T_{e}L_{g}(\frakd^{\bot})$. Moreover, the Lie algebra projectors satisfy $                \mathfrak{P} = T_{g}L_{g^{-1}} \circ \mathcal{P} \circ T_{e}L_{g}$ and  $\mathfrak{Q} = T_{g}L_{g^{-1}} \circ \mathcal{Q} \circ T_{e}L_{g}$.

%The following lemma from \cite{stratoglou2023virtual} will be useful in the sequel:
%\begin{lemma}
 %   Let $\D$ be a left-invariant distribution. Given a left invariant metric on $G$ consider the associated orthogonal distribution $\D^{\bot}$ and the orthogonal projectors $\mathcal{P}: TG \to \D$ and $\mathcal{Q}:TG\to \D^{\bot}$. Then the following statements hold:
  %  \begin{enumerate}
   %     \item The orthogonal distribution $\D^{\bot}$ is left-invariant and $\D^{\bot}_{g} = T_{e}L_{g}(\frakd^{\bot})$.
    %    \item The Lie algebra projectors satisfy
     %       \begin{equation}\label{Lie algebra projections}
      %          \mathfrak{P} = T_{g}L_{g^{-1}} \circ \mathcal{P} \circ T_{e}L_{g} \text{ and } \mathfrak{Q} = T_{g}L_{g^{-1}} \circ \mathcal{Q} \circ T_{e}L_{g}.
       % \end{equation}
    %\end{enumerate}
%\end{lemma}

%\begin{proof}
 %   Cite our Lie group paper.
%\end{proof}

Next, we define a nonholonomic $\frakd$-connection $\tilde{\nabla}^{\frakd}: \g \times \g \to \g$ to be a bilinear map satisfying 
\begin{equation}\label{d:connection}
    \tilde{\nabla}^{\frakd}_{\xi}\eta = \left(\tilde{\nabla}^{nh}_{\xi_{L}} \eta_{L}\right)(e),
\end{equation}
where $\tilde{\nabla}^{nh}$ is the nonholonomic connection corresponding to the Levi-Civita connection $\tilde{\nabla}$ on $G$. Then
    $$\tilde{\nabla}^{\frakd}_{\xi}\eta = \tilde{\nabla}^{\g}_{\xi} \eta + (\tilde{\nabla}^{\g}_{\xi}\mathfrak{Q})(\eta),$$ where $\tilde{\nabla}^{\g}_{\xi}\mathfrak{Q}$ is the bilinear map $(\tilde{\nabla}^{\g}_{\xi}\mathfrak{Q})(\eta)=\tilde{\nabla}^{g}_{\xi}(\mathfrak{Q}(\eta)) - \mathfrak{Q} (\tilde{\nabla}_{\xi} \eta)$, where $\tilde{\nabla}^\g$ is the Riemannian $\g$-connection corresponding to the Levi-Civita connection $\tilde{\nabla}$ on $G$.

    %\textcolor{red}{is this information enough for the nonholonomic connection? we don't need projections? (projections whose restrictions at $\g$ give the (4.1) projections). Or better cite Bullo and Lewis 2005}.

%\begin{proposition}
 %   Define the bilinear map $(\nabla^{\g}_{\xi}\mathfrak{Q})(\eta)=\nabla^{g}_{\xi}(\mathfrak{Q}(\eta)) - \mathfrak{Q} (\nabla_{\xi} \eta)$, where $\nabla^\g$ is the Riemannian $\g$-connection corresponding to the Levi-Civita connection $\nabla$ on $G$. Then
  %  $$\nabla^{\frakd}_{\xi}\eta = \nabla^{\g}_{\xi} \eta + (\nabla^{\g}_{\xi}\mathfrak{Q})(\eta).$$
%\end{proposition}

It follows that $\tilde{\nabla}^{\frakd}_\xi \eta = \mathfrak{P}(\tilde{\nabla}_{\xi}^\g\eta)$, for all $\xi, \eta \in \frakd$. Therefore, we obtain the explicit expression for $\xi,\eta \in \frakd$
\begin{equation}\label{hcon_decomp}
    \tilde{\nabla}^{\frakd}_{\xi} \eta =\frac12 \mathfrak{P}\left([\xi,\eta]_\g-\sharp \left[ \text{ad}^{*}_{\xi} \flat (\eta)\right] - \sharp \left[ \text{ad}^{*}_{\eta} \flat (\xi)\right] \right).\end{equation}

We now express nonholonomic trajectories on Lie groups in terms of the Riemannian $\frakd$-connection (see \cite{stratoglou2023virtual} for details). Suppose that $g: [a,b] \to G$ is a nonholonomic trajectory with respect to a left-invariant metric and distribution $\D$ and let $\xi(t) = g(t)^{-1} \dot{g}(t)$. Then, $\xi$ satisfies
    \begin{equation}\label{reduced trajectory}
        \dot{\xi} + \tilde{\nabla}_{\xi}^{\frakd} \xi = 0,
    \end{equation}
    or, equivalently,
    \begin{equation}
        \dot{\xi} + (\mathfrak{P} \circ \sharp) \left[ \text{ad}^{*}_{\xi(t)}\flat(\xi(t))\right] = 0, \quad \xi\in \frakd.
    \end{equation}

%%%%%%%%%%%%%%%%%%%%%%%%%%%%%%%%%%%%%%%%%%%%%%%%%%%%%%%%%%

\section{Nonholonomic systems on Riemannian homogeneous spaces}\label{sec3}
\subsection{Homogeneous spaces}

Let $G$ be a connected Lie group equipped with a left-invariant Riemannian metric. A \textit{homogeneous space} $H$ of $G$ is a smooth manifold on which $G$ acts transitively. Any Lie group is itself a homogeneous space, where the transitive action is given by left-translation (or right-translation). 

Suppose that $\Phi: G \times H \to H$ is a transitive left-action, which we denote by $gx := \Phi_g(x)$. It can be shown that for any $x \in H$, we have $G/\text{Stab}(x) \cong H$ as differentiable manifolds, where $\text{Stab}(x) := \{g \in G \ \vert \ gx = x\}$ denotes the \textit{stabilizer subgroup} (also called the \textit{isotropy subgroup}) of $x$, and $G/\text{Stab}(x)$ denotes the space of equivalence classes determined by the equivalence relation $g \sim h$ if and only if $g^{-1}h \in \text{Stab}(x)$. In addition, for any closed Lie subgroup $K \subset G$, the left-action $\Phi: G \times G/K \to G/K$ satisfying $\Phi_g([h]) = [gh]$ for all $g, h \in G$ is transitive, and so $G/K$ is a homogeneous space. Hence, we may assume without loss of generality that $H := G/K$ is a homogeneous space of $G$ for some closed Lie subgroup $K$. Let $\pi: G \to H$ be the canonical projection map. With this notation, the left multiplication commutes with the left action in the sense that $\Phi_{g}\circ \pi=\pi\circ L_{g}$ for any $g\in G$.

\vspace{.2cm}

\begin{example}\label{example1}
Consider the special Euclidean group $SE(3)$ composed of rotations and translations in the space $\R^{3}$. An element of $SE(3)$ is a pair $(R,r)$ where $R$ is a rotation matrix in $SO(3)$ and $r\in \R^{3}$. Consider the action of $\SE(3)$ on $\R^3$ given by the map $\Psi:\SE(3)\times\R^{3}\to\R^{3},$ $ \Psi_{(R,r)}(x)=Rx+r$. This action performs a rotation $R$ on the vector $x\in\R^{3}$ followed by a translation by $r$. This action is transitive, thus $\R^3$ has the structure of a homogeneous space. 

That is, $\R^3$ can be seen as the quotient $\R^3\simeq\SE(3)/K$ where $K$ is a subgroup of $\SE(3)$ such that $K=\{(k,0) : k\in\SO(3)\}$ thus $K\simeq\SO(3)$. Moreover, $K=\text{Stab}(0)$ for $0\in\R^3$ which is the stabilizer subgroup of the action $\Psi.$ The projection map $\pi:\SE(3)\to\R^3$ is given by $\pi(R,r)=R0+r=r.$

Moreover, we have that $\SE(3)\simeq\SO(3)\times\R^3$, using the hat map identification $\hat{(\cdot)}:\R^3 \to \so(3)$ and the standard basis $\{e_1,e_2,e_3\}$ for $\R^3$ (see \cite{holm2009geometric} for instance), we have that the basis for the Lie algebra $\se(3)$ consists of the elements $\bar{e}_1=(0,e_1), \bar{e}_2=(0,e_2), \bar{e}_3=(0,e_3), \bar{e}_4=(\hat{e}_1,0), \bar{e}_5=(\hat{e}_2,0), \bar{e}_6=(\hat{e}_3,0)$. Hence, from the projection map $\pi$, we deduce $T_{I}\pi (\bar{e}_1)=e_1, T_{I}\pi(\bar{e}_2)=e_2, T_{I}\pi(\bar{e}_3)=e_3, T_{I}\pi(\bar{e}_4)=T_{I}\pi(\bar{e}_5)=T_{I}\pi(\bar{e}_6)=0$, where $I$ denotes the $4\times 4$ identity matrix.  \hfill$\square$  
\end{example}

We define the \textit{vertical subspace} at $g \in G$ by $\text{Ver}_g := \ker (T_{g} \pi)$, from which we may construct the \textit{vertical bundle} as $VG := \bigsqcup_{g \in G} \{g\} \times \text{Ver}_g$. Given a Riemannian metric $\left< \cdot, \cdot\right>_G$ on $G$, we define the \textit{horizontal subspace} at any point $g \in G$ (with respect to $\left< \cdot, \cdot \right>_G)$ as the orthogonal complement of $\text{Ver}_g.$ That is, $\text{Hor}_g := \text{Ver}_g^\perp$. Similarly, we define the \text{horizontal bundle} as $HG := \bigsqcup_{g \in G} \{g\} \times \text{Hor}_g$. Both the vertical and horizontal bundles are vector bundles, and are in fact subbundles of the tangent bundle $TG$. It is clear that $T_g G = \text{Ver}_g \oplus \text{Hor}_g$ for all $g \in G$, so that the Lie algebra $\g$ of $G$ admits the decomposition $\g = \mathfrak{s} \oplus \mathfrak{h}$, where $\mathfrak{s}$ is the Lie algebra of $K$ and $\mathfrak{h} \cong T_{\pi(e)} H$. We denote the orthogonal projections onto the vertical and horizontal subspaces by $\mathcal{V}$ and $\mathcal{H}$. Moreover, both vertical and horizontal spaces are left-invariant, that is $\text{Ver}_{g}=T_{e}L_{g}(\mathfrak{s})$ and $\text{Hor}_{g}=T_{e}L_{g}(\mathfrak{h})$, provided that the Riemannian metric is left-invariant.

%\todo{The notation for the ser of vector fields and the set of sections was not introduced before}

A section $Z \in \Gamma(HG)$ is called a \textit{horizontal vector field}. That is, $Z \in \mathfrak{X}(G)$ and $Z(g) \in \text{Hor}_g$ for all $g \in G.$ A vector field $Y \in \mathfrak{X}(G)$ is said to be $\pi$-related to some $X \in \mathfrak{X}(H)$ if $(T_{g}\pi) (Y_g) = X_{\pi(g)}$ for all $g \in G$. If, in addition, $Y \in \Gamma(HG),$ we say that $Y$ is a \textit{horizontal lift} of $X$. We further define a horizontal lift of a smooth curve $q: [a, b] \to H$ as a smooth curve $\tilde{q}: [a, b] \to G$ such that $\pi \circ \tilde{q} = q$ and $\dot{\tilde{q}}(t)$ is horizontal for all $t \in [a, b]$. %\textcolor{red}{Picture???} %We have the following results (see \cite{goodman2022reduction}) 

Let $H$ be a homogeneous space of $G$ and $X \in \mathfrak{X}(H)$. In \cite{goodman2024reduction}, the authors have shown that for each $X \in \mathfrak{X}(H),$ there exists a unique horizontal lift $\tilde{X}$ of $X$. That is, the map $\tilde{\cdot}: \mathfrak{X}(H) \to \Gamma(HG)$ sending $X \mapsto \tilde{X}$ is $\R$-linear and one-to-one. Moreover, for each smooth curves $q: [a, b] \to H$ and $q_0 \in \pi^{-1}(\{q(a)\})$, there exists a unique horizontal lift $\tilde{q}: [a, b] \to G$ of $q$ satisfying $\tilde{q}(a) = q_0$, called the \textit{horizontal lift} of $q$ whose velocity curve $\dot{\tilde{q}}$ is $\pi$-related to $\dot{q}$ and is horizontal. If $q: [a, b] \to H$ and $\tilde{q}: [a, b] \to G$ is a horizontal lift of $q$, then, for each $\tilde{\eta}: [a, b] \to \mathfrak{h}$, there exists a unique $X \in \Gamma(q)$ such that its horizontal lift $\tilde{X}$ along $\tilde{q}$ satisfies $(T_{\tilde{q}(t)}L_{\tilde{q}(t)^{-1}}) (\tilde{X}(t)) = \tilde{\eta}(t)$ for all $t\in[a,b]$.

\subsection{Riemannian Homogeneous Spaces}

Consider a connected Lie group $G$ and a homogeneous space $H = G/K$ of $G$. Since $H$ is a smooth manifold, it can be equipped with a Riemannian metric. Similarly to Section \ref{sec: background_Lie}, we are interested in those metrics $\left<\cdot, \cdot\right>_H$ which in some sense preserve the structure of the homogeneous space. In this case, we wish to choose $\left<\cdot, \cdot\right>_H$ so that the canonical projection map $\pi: G \to H$ is a \textit{Riemannian submersion}. That is, so that $T_{g}\pi$ is a linear isometry between $\text{Hor}_g$ and $T_{\pi(g)} H$ for all $g \in G$. In such a case, we call $H$ a \textit{Riemannian homogeneous space}. It is clear that if $H$ is a Riemannian homogeneous space, then $\left<\Hor(X), \Hor(Y)\right>_G = \left<(T_{g}\pi) (X), (T_{g}\pi) (Y)\right>_H$ for all $X, Y \in T_g G, g \in G$. In particular, $\big{\langle} \tilde{X}, \tilde{Y} \big{\rangle}_G = \big{\langle} X, Y \big{\rangle}_H$ for all $X, Y \in T_g G, g \in G$. The metric $\left< \cdot, \cdot \right>_H$ is said to be \textit{$G$-invariant} if it is invariant under the left-action $\Phi_g$ for all $g \in G$. It can be shown that every homogeneous space $H = G/K$ that admits a $G$-invariant metric is \textit{reductive}. That is, the Lie algebra admits a decomposition $\g = \mathfrak{s} \oplus \mathfrak{h}$, where $\mathfrak{s}$ is the Lie algebra of $K$, and $\mathfrak{h}$ satisfies $[\mathfrak{s}, \mathfrak{h}] \subset \mathfrak{h}$. In particular, this implies that $\mathfrak{h} \cong T_{\pi(e)}(G/K)$ as vector spaces.

\vspace{.2cm}

\begin{example}\label{example2}
Continuing with Example \ref{example1}, equip the Lie group $\SE(3)$ with the usual left-invariant metric determined by the inner product $\langle(\hat{\Omega}_1,r_1),(\hat{\Omega}_2,r_2)\rangle_{\se(3)}=\Omega_1^T\Omega_2 + r_1^Tr_2$, where $\Omega_1, \Omega_2\in\R^3$ and $r_1, r_2\in\R^3$. Using this metric we define an inner product on $T_0\R^3$ through the relation $\langle (T_{I}\pi) (\xi_{1}), (T_{I}\pi) (\xi_{2}) \rangle_{T_0\R^3}=\langle \xi_{1}, \xi_{2}\rangle_{\se(3)}$ for all $\xi_{1}, \xi_{2}\in \se(3)$ and we extend this inner product to an $\SE(3)$-invariant Riemannian metric on $\R^3$ by $\langle X,Y\rangle_{\R^3}=\langle (T_{r}\Phi_{(0,r)})(X),(T_{r}\Phi_{(0,r)})(Y)\rangle_{T_0\R^3}$ where $X,Y\in T_r\R^3$. Thus, since $T_{r}\Phi_{(0,r)}=Id$, we have that $\langle X,Y\rangle_{\R^3}=X^{T}Y$, which is the standard Euclidean product on $\R^{3}$.

With this Riemannian structure, we have that $\s=\ker(T_{I}\pi)=\text{span}\{\bar{e}_4, \bar{e}_5, \bar{e}_6\}\simeq\so(3)$ and we define $\h=\s^\perp$ such that $\h=\text{span}\{\bar{e}_1, \bar{e}_2, \bar{e}_3\}\simeq\R^3.$ The adjoint operator for the Lie algebra $\se(3)$ is given by $\ad:\se(3)\times\se(3)\to\se(3),$ 
\[\ad_{(\hat{\Pi},t)}(\hat{\Omega},s)=(\ad_{\hat{\Pi}}\hat{\Omega}, \hat{\Pi}s - \hat{\Omega}t),\]
where $\ad_{\hat{\Pi}}\hat{\Omega}$ is the adjoint operator on $\so(3)$ and $t,s\in\R^3$ (see \cite{holm2009geometric} for instance).

Since $\s=\text{span}\{\bar{e}_4, \bar{e}_5, \bar{e}_6\}\simeq\so(3)$, the vertical space of $\SE(3)$ is given by $\text{Ver}_g=\text{span}\{g\bar{e}_4, g\bar{e}_5, g\bar{e}_6\}\simeq\SO(3)$ and the horizontal space is $\text{Hor}_g=\text{span}\{g\bar{e}_1, g\bar{e}_2, g\bar{e}_3\}\simeq\R^3$ where $g\in\SE(3)$. The horizontal projection is given by $\Hor(\hat{\Omega},r)=(0,r).$\hfill$\square$
\end{example}

There is an equivalence between the existence of a $G$-invariant metric on $H$ and the existence of a left-invariant metric on $G$ for which $H$ is a Riemannian homogeneous space, for more details see \cite{goodman2024reduction}. 

Denote the Levi-Civita connections on $H$ and $G$ with respect to these metrics by $\nabla$ and $\tilde{\nabla}$, respectively. In a Riemannian homogeneous space, due to the fact that $\pi$ is a Riemannian submersion we have the remarkable property that if a geodesic on $G$ is horizontal at some point, then it is horizontal at all points. In particular, the horizontal lift of geodesics on $H$ are geodesics on $G$ (see \cite{o1967submersions}). These geodesics are usually called \textit{horizontal geodesics}.

\vspace{.2cm}

\begin{theorem}\label{thm:LP_geo}
Let $H$ be a Riemannian homogeneous space with respect to a Lie group action by $G$. If $g: [a, b] \to G$ is a curve on $G$, $q(t)=\pi(g(t))$ the projection of $g$ on $H$ and $\xi(t) := (T_{g(t)} L_{g^{-1}(t)}) (\dot{g}(t))$, then the following statements are equivalent:
\begin{enumerate}
    \item The curve $g: [a, b] \to G$ is a horizontal geodesic,
    \item $\xi(t)\in\h$ for all $t\in [a, b]$ and
    \begin{equation}\label{geodesics:thm:LP_geo}
      \dot{\xi} + \tilde{\nabla}_{\xi}^{\g} \xi=0.
    \end{equation}
    \item $g(t)=\tilde{q}(t)$ and the curve $q$ is a geodesic for $\nabla$.
\end{enumerate}
\end{theorem}

\begin{proof}
    The proof that statements (1) and (2) are equivalent derives from equation \eqref{EP: geo} and the fact that geodesics horizontal at one point remain horizontal for all times. If $g(t)$ is a horizontal geodesic with respect to the Levi-Civita connection, then $$(T_{e} L_{g}) \left( \dot{\xi} + \tilde{\nabla}_{\xi}^{\g} \xi\right) = 0, \quad \dot{g}\in HG.$$ Since left-translation is a diffeomorphism, we have the desired result.

    The equivalence between (1) and (3) is a general fact for Riemannian submersions that can be found on \cite{o1967submersions}.
\end{proof}

%\textcolor{red}{Theorem 2 in the context of the example}

\begin{example}
Continuing with Examples \ref{example1} and \ref{example2}, consider a horizontal geodesic $g:[a,b]\to\SE(3)$ given by $g(t)=(R,r(t))$ where $R$ is a constant element of $\SO(3)$ and $r(t)\in\R^3$. The left translation to the identity gives $\xi=(T_{g}L_{g^{-1}})(\dot{g})=(R^{T}, -R^{T}r(t))(0,\dot{r}(t))=\left(0,R^{T}\dot{r}\right)$ thus $\xi$ is an element of the horizontal subspace at the identity, i.e. $\xi\in\h=\emph{Hor}_{e}=\text{span}\{\bar{e}_1,\bar{e}_2,\bar{e}_3\}$. From Lemma \ref{Riemanian:g:conn}, $\tilde{\nabla}^\g_{\xi}\xi=0$, since $\text{ad}_{\xi}^{*}\xi = 0$. Thus,  (\ref{geodesics:thm:LP_geo}) reads $\ddot{r}=0.$ Let $q(t)=\pi(g(t))$. So,  $q(t)=r(t)\in\R^3$. Clearly, since $\ddot{r}=0$, we have that $r$ is a geodesic on $\R^{3}$. In addition, the horizontal lift of $q(t)$ to the point $(R,r(0))\in \SE(3)$ is $\tilde{q}(t)=(R, r(t))$.\hfill$\square$
\end{example}

The next Proposition relates the covariant derivative of vector fields in $G$ and $H$ and will be useful later.

\vspace{.2cm}

\begin{proposition}[Lemma 45 from Ch. 7, \cite{o1983semi}]\label{pi:related:connections}
    Let $\nabla$ and $\tilde{\nabla}$ be the Levi-Civita connections on $H$ and $G$, respectively. Then,
    \begin{enumerate}
        \item $T\pi(\tilde{\nabla}_{\tilde{X}}\tilde{Y}) = \nabla_{X} Y, \ X,Y\in \mathfrak{X}(H).$
        \item $\mathcal{H}\tilde{\nabla}_{\tilde{X}}\tilde{Y} = \widetilde{\nabla_{X}Y} \ $, where the right-hand side denotes the horizontal lift of the $\nabla_{X} Y$.
    \end{enumerate}
\end{proposition}

\subsection{Mechanical systems on Riemannian homogeneous spaces}
For the rest of the paper, assume that $H$ is a homogeneous manifold acted on by the Lie group $G$ and both manifolds are equipped with Riemannian metrics making $\pi:G\rightarrow H$ a Riemannian submersion.

Let $V:H\rightarrow \R$ be a potential function on the homogeneous space $H$. Via the projection map $\pi$, this potential function induces a potential function on the Lie group $G$ denoted by $\tilde{V}=V\circ \pi$.
% given by $L(g,\dot{g})=\frac{1}{2}\langle \dot{g}, \dot{g} \rangle - \tilde{V}(g)$

Let $q:[a,b]\rightarrow H$ be a trajectory of a mechanical system with Lagrangian function $L:TH\rightarrow \R$ of the form \eqref{mechanical:lagrangian}. Then, the curve $q$ satisfies the equation
\begin{equation}\label{mechanical:equation}
    \nabla_{\dot{q}} \dot{q}(t) = -\text{grad} \ V (q(t)),
\end{equation}
where $\text{grad}$ is the gradient with respect to the metric on $H$.

The next result establishes that the gradient vector field $\widetilde{\text{grad }} \tilde{V}$, where $\widetilde{\text{grad }}$ is the gradient with respect to the metric on $G$, is a horizontal vector field.
\vspace{.3cm}
\begin{lemma}\label{grad:horizontal}
    If $\tilde{V}=V\circ \pi$ is the potential function induced by $V:H\rightarrow \R$, then the vector field $\widetilde{\text{grad }} \tilde{V}\in \Gamma(HG)$, where $\widetilde{\text{grad }}$ is the gradient with respect to the metric on $G$.
\end{lemma}
\vspace{-.5cm}
\begin{proof}
    By construction, since given a vertical vector field $Y\in VG$, we have that
$$\langle \widetilde{\text{grad }} \tilde{V}, Y \rangle = d\tilde{V}(Y)  = dV(T\pi(Y))=0.$$
%\vspace{-.2cm}
\end{proof}
\vspace{-.5cm}
Due to the fact that the gradient vector field of the potential $\tilde{V}$ is horizontal, one might deduce that the horizontal lift $g=\tilde{q}$ of a curve $q$ satisfying \eqref{mechanical:equation} satisfies the mechanical equation
\begin{equation}\label{Lie:mechanical:equation}
    \tilde{\nabla}_{\dot{g}} \dot{g} = -\widetilde{\text{grad }} \tilde{V}(g(t)).
\end{equation}

%$\text{grad}_{G} \ \tilde{V} (g)= g \cdot \text{grad}_{G} \ \tilde{V} (e)$ and also $\pi_{*}\left(\text{grad}_{G} \ \tilde{V} (e)\right) = \text{grad}_H \ V (\pi(e))$. Thus, we deduce that

Writing this equation entirely in the Lie algebra of $G$ is often impossible in real applications since the potential function $\tilde{V}$ typically does not possess any invariance property with respect to the group multiplication. Interesting applications in geometric control occur when the potential possesses partial symmetries \cite{goodman2024b-reduction}, \cite{colombo2023lie}, \cite{bloch2017optimal}.

%\todo{put the last paragraph as a Lemma and show the result}

\vspace{.2cm}

\begin{theorem}
    \label{thm:LP_mecg}
    Let $H$ be a Riemannian homogeneous space w.r.t a Lie group action by $G$ and $V:H\to \R$ a potential function. If $g: [a, b] \to G$ is a curve on $G$, $q(t)=\pi(g(t))$ the projection of $g$ on $H$, and $\xi(t) := (T_{g(t)} L_{g^{-1}(t)}) (\dot{g}(t))$, then the following statements are equivalent:
\begin{enumerate}
    \item The curve $g: [a, b] \to G$ is a horizontal trajectory of the mechanical system \eqref{Lie:mechanical:equation}, that is $\dot{g}(t)\in HG$ for all $t\in[a, b]$.
    \item $\xi\in \h$ on $[a, b]$ and satisfies
\begin{align}
    \dot{\xi} + \tilde{\nabla}_{\xi}^\g \xi = -(T_{g} L_{g^{-1}})\left(\widetilde{\text{grad }} \tilde{V}  (g(t))\right).\label{LP: mech}
\end{align}
\item If $q(t)=\pi(g(t))$ then $g(t)=\tilde{q}(t)$ and the curve $q$ satisfies \eqref{mechanical:equation}.
\end{enumerate}

\end{theorem}

\begin{proof}
    The equivalence of the statements (1) and (2) is a direct consequence of \eqref{Cov-to-cov} and of equation \eqref{Lie:mechanical:equation}. Indeed,
    $$\tilde{\nabla}_{\dot{g}} \dot{g} = T_{e}L_{g}\left( \dot{\xi} + \tilde{\nabla}_{\xi}^\g \xi\right)$$
    from where it is clear that $\tilde{\nabla}_{\dot{g}} \dot{g} = -\widetilde{\text{grad }} \tilde{V}(g(t))$ if and only if $\dot{\xi} + \tilde{\nabla}_{\xi}^\g \xi = -(T_{g} L_{g^{-1}})\left(\widetilde{\text{grad }} \tilde{V}  (g(t))\right)$ since $$(T_{e}L_{g})(T_{g} L_{g^{-1}})\left(\widetilde{\text{grad }} \tilde{V}  (g(t))\right) = \widetilde{\text{grad }} \tilde{V}  (g(t)).$$
    
    The equivalence of the statements (1) and (3) can be seen firstly from the fact that if $q(t)=\pi(g(t))$ then $g(t)=\tilde{q}(t)$ if and only if $g$ is horizontal. 
    
    Secondly, using Lemma \ref{grad:horizontal} stating that $\widetilde{\text{grad }} \tilde{V}$ is a horizontal vector field and since by construction $\tilde{V}=V\circ\pi$, we conclude that $(T\pi)(\widetilde{\text{grad }} \tilde{V}) = \text{grad } V$. Thus, $\widetilde{\text{grad }} \tilde{V}$ is the horizontal lift of $\text{grad } V$. Thus, if $g$ satisfies equation \eqref{Lie:mechanical:equation}, then, by Proposition \ref{pi:related:connections}(2), the horizontal lift of $\nabla_{\dot{q}} \dot{q}(t)$ must coincide with the horizontal lift of $-\widetilde{\text{grad }} \tilde{V}$. Hence, equation \eqref{mechanical:equation} must hold.
    
    Conversely, suppose that equation \eqref{mechanical:equation} holds. %then we may take the horizontal lift of both sides of the equation and using Proposition \ref{pi:related:connections}(2), we conclude that
    %$$\mathcal{H} \left(\tilde{\nabla}_{\dot{g}} \dot{g} + \widetilde{\text{grad }} \tilde{V}(g(t)) \right) = 0.$$
    By Theorem \ref{thm:LP_geo}, we know that the geodesic vector field of $\tilde{\nabla}$ is tangent to $HG$, since geodesics with initial velocity in $HG$, remain in $HG$ for all time. Furthermore, from \cite{stratoglou2023virtual}, the vector field whose trajectories are the solution of equation \eqref{Lie:mechanical:equation} has the form
    $$\tilde{\Gamma} = G - (\widetilde{\text{grad }} \tilde{V}(g(t)))^{\textbf{v}},$$
    where $G$ is the geodesic vector field of $\tilde{\nabla}$ and $(\widetilde{\text{grad }} \tilde{V}(g(t)))^{\textbf{v}}$ is the vertical lift of the vector field $\widetilde{\text{grad }} \tilde{V}(g(t))$. Thus, $\Gamma$ is the sum of two vector fields that are tangent to $HG$, implying that $\Gamma$ is itself tangent to $HG$. Consequently, if the trajectories of $\Gamma$, that is, of equation \eqref{Lie:mechanical:equation}, are horizontal at one point, they are horizontal at all points. In particular, since the trajectories of equation \eqref{Lie:mechanical:equation} project onto trajectories of equation \eqref{mechanical:equation}, we msut have that the horizontal lift $g = \tilde{q}$ must be a solution of \eqref{Lie:mechanical:equation}.
\end{proof}

%\subsection{Invariant distributions on $H$}

\subsection{Nonholonomic systems on Riemannian homogeneous spaces}
A distribution $\mathcal{D}$ on $H$ is said to be $G$-invariant if $\mathcal{D}_{g\cdot \pi(e)}=(\Phi_{g})_{*}(\D_{\pi(e)})$. We may consider the horizontal lift of the distribution $\mathcal{D}$ to $G$. Given $g\in G$ such that $\pi(g)=q$, we have that
$\tilde{\mathcal{D}}_{g} = \{ \tilde{v}_{g}\in \text{Hor}_{g} \ |  (T_{g}\pi)(\tilde{v}_{g})\in \mathcal{D}_{q}\ \}$. Essentially, $\tilde{\mathcal{D}}_{g}$ is a subspace of $\text{Hor}_{g}$ composed of the horizontal lift of vectors in $\mathcal{D}_{q}$.

\vspace{.1cm}

\begin{proposition}
    The distribution $\tilde{\mathcal{D}}_{g}$ is left-invariant.
\end{proposition}

\vspace{-.5cm}

\begin{proof}
Let $\frakd = \tilde{\mathcal{D}}_{e}$. Note that $T_{e}\pi (\frakd) = \D_{\pi(e)}$ by definition and also that, by  $G$-invariance of $\D$, we have that $\D_{\pi(g)} = (T_{\pi(e)}\Phi_{g})(\D_{\pi(e)})$ for all $g \in G$. Thus $\D_{\pi(g)}=(T_{e}(\Phi_{g}\circ \pi))(\frakd)$. In addition, since $\pi$ commutes with the action $\Phi_{g}$ and the left action, we also have that $\D_{\pi(g)} = (T_{e}(\pi\circ L_{g}))(\frakd)$. Equivalently,
$$\D_{\pi(g)} = T_{g}\pi (T_{e}L_{g} (\frakd)).$$
But notice that $T_{e}L_{g} (\frakd)\in \text{Hor}_{g}$ since $HG$ is left invariant, and that $T_{g}\pi|_{\text{Hor}_{g}}$ maps $\tilde{\mathcal{D}}_{g}$ isomorphically to $\D_{\pi(g)}$. Both these facts imply that $\tilde{\mathcal{D}}_{g}=T_{e}L_{g} (\frakd)$.\end{proof}

As a consequence of the previous proposition, there exists a subspace $\frakd$ of the Lie algebra $\g$ such that
$\tilde{\mathcal{D}}_{g}=T_{e}L_{g}(\frakd)$.

Throughout this section we will consider a $G$-invariant distribution $\D$ on $H$ and since the horizontal lift of $\D$ is a left-invariant distribution $\tilde{\mathcal{D}}$ on $G$,  $\frakd$ will be the restriction to the identity of $\tilde{\mathcal{D}}$. The orthogonal complement of $\frakd$ with respect to the inner product on $\g$ will be denoted by $\frakd^{\bot}$. Note that $\frakd^{\bot}$ has non-zero intersections with both the horizontal space $\h=H_{e}G$ and the vertical space $\mathfrak{s}=\text{Ver}_{e}$, while $\frakd\subseteq \h$.

Consider on $G$, the orthogonal projections $\tilde{P}:TG\to \tilde{\mathcal{D}}$ and $\tilde{Q}:TG\to \tilde{\mathcal{D}}^{\bot}$. With this projections, we are able to define the nonholonomic connection $\tilde{\nabla}^{nh}$ given by the analogous expression to that in \eqref{nh:connection}:
$$\tilde{\nabla}^{nh}_X Y =\tilde{\nabla}_{X} Y + (\tilde{\nabla}_{X} \tilde{Q})(Y).$$ Similarly, on the manifold $H$, we can define the nonholonomic connection $\nabla^{nh}$ with respect to the Riemannian connection $\nabla$ and the orthogonal projections $\mathcal{P}:TH\to \D$ and $\mathcal{Q}:TH\to \D^{\bot}$. We have the following results relating geodesics with respect to both nonholonomic connections.

\vspace{.2cm}

\begin{lemma}[\cite{bullo2019geometric}, Section 2]\label{Lewis:lemma}
    Given a Riemannian manifold $Q$, letting $\nabla$ be the Levi-Civita connection and $\D$ a non-integrable distribution then a curve $q:[a,b]\to Q$ is a geodesic of the nonholonomic connection $\nabla^{nh}$ if and only if
    $$\nabla_{\dot{q}}\dot{q} \in \D^{\bot} \text{ and } \dot{q} \in \D.
    $$
\end{lemma}

Then, we have the following result:

\vspace{.2cm}

\begin{lemma}\label{nh:trajectory:lift}
    A curve $q:[a,b]\to H$ is a geodesic associated with $\nabla^{nh}$ and the constraint distribution $\mathcal{D}$ if and only if its tangent lift $\tilde{q}$ is a geodesic with respect to $\tilde{\nabla}^{nh}$ and the constraint distribution $\tilde{\mathcal{D}}$.
\end{lemma}

\begin{proof}
    Suppose $q(t), t\in[a,b]$ is a curve in $H$ such that its tangent lift $\tilde{q}$ is a geodesic with respect to $\tilde{\nabla}^{nh}$ and the constraint distribution $\tilde{\mathcal{D}}$. Thus $\tilde{\nabla}^{nh}_{\dot{\tilde{q}}}\dot{\tilde{q}}=0$ and $\dot{\tilde{q}}\in\tilde{\mathcal{D}}$.
    
    Then, by Lemma \ref{Lewis:lemma}, we have that
    $\tilde{\nabla}_{\dot{\tilde{q}}}\dot{\tilde{q}}\in \tilde{\mathcal{D}}^{\bot}$ and $\dot{\tilde{q}}\in\tilde{\mathcal{D}}$. Using Proposition \ref{pi:related:connections} and also $T\pi(\tilde{\mathcal{D}}^{\bot})=\D^{\bot}$, we conclude that $\nabla_{\dot{q}}\dot{q}\in \mathcal{D}^{\bot}$ and $\dot{q}\in \D$. Finally, using Lemma \ref{Lewis:lemma} again, we deduce that $q$ satisfies $\nabla^{nh}_{\dot{q}}\dot{q}=0$.

    Conversely, if $q$ satisfies $\nabla^{nh}_{\dot{q}}\dot{q}=0$ then, by Lemma \ref{Lewis:lemma}, it also satisfies $\nabla_{\dot{q}}\dot{q}\in \mathcal{D}^{\bot}$ and $\dot{q}\in \D$. Therefore, the horizontal lift of $\nabla_{\dot{q}}\dot{q}$, denoted by $\widetilde{\nabla_{\dot{q}}\dot{q}} \ $ belongs to $\tilde{\mathcal{D}}^{\bot}$. But since by the second statement of Proposition \ref{pi:related:connections}, we have that $\tilde{\nabla}_{\dot{\tilde{q}}}\dot{\tilde{q}} = \widetilde{\nabla_{\dot{q}}{\dot{q}}}+V$
    with $V\in VG$, and noting that $VG\subseteq \tilde{\mathcal{D}}^{\bot}$, we must have that $\tilde{\nabla}_{\dot{\tilde{q}}}\dot{\tilde{q}} \in \tilde{\mathcal{D}}^{\bot}$. Obviously, we have also that $\dot{\tilde{q}}\in \tilde{\mathcal{D}}$.% Hence, we obtain the result.
\end{proof}

We have the following result relating corresponding nonholonomic trajectories in each space.

\vspace{.2cm}

\begin{theorem}
    Let $H$ be a Riemannian homogeneous space with respect to a Lie group action by $G$, $\tilde{\mathcal{D}}\subseteq HG$ a left-invariant distribution on $G$ and $V:H\to \R$ a potential function. If $g: [a, b] \to G$ is a curve on $G$, $q(t)=\pi(g(t))$ the projection of $g$ on $H$ and $\xi(t) := (T_{g(t)} L_{g^{-1}(t)}) (\dot{g}(t))$, then the following statements are equivalent:
    \begin{enumerate}
        \item $g:[a,b]\to G$ is a nonholonomic trajectory associated to the mechanical Lagrangian \eqref{mechanical:lagrangian} and the distribution $\tilde{\mathcal{D}}$.
        \item $\xi\in \frakd$ on $[a,b]$ and satisfies
    $$\dot{\xi} + \tilde{\nabla}_{\xi}^{\frakd} \xi = -\mathfrak{P}\left((T_{g} L_{g^{-1}})\left(\widetilde{\text{grad }} \tilde{V}  (g(t))\right) \right),$$
    where $\mathfrak{P}:\mathfrak{g}\to \frakd$ is the projection associated with the decomposition $\g=\frakd \oplus \frakd^{\bot}$ and $\tilde{\nabla}^{\frakd}$ is the $\frakd$-connection associated with $\tilde{\nabla}^{\g}$ and $\frakd$ defined by \eqref{d:connection}.
        \item If $q(t)=\pi(g(t))$, then $\dot{q}\in \D=T\pi(\tilde{\mathcal{D}})$, $g(t)=\tilde{q}(t)$ and $q$ satisfies equations \eqref{nonholonomic:mechanical:equation} associated with the $G$-invariant metric on $H$ and the distribution $\D$.
    \end{enumerate} 
\end{theorem}

\begin{proof}
    We first prove the equivalence of the statements (1) and (2) and later the equivalence between (1) and (3). If (1) holds, then the curve $g$ satisfies
    $$\tilde{\nabla}^{nh}_{\dot{g}} \dot{g} = \tilde{P}(\tilde{\nabla}_{\dot{g}} \dot{g})$$
    which implies that $g$ satisfies
    \begin{equation*}
    \tilde{P}(\tilde{\nabla}_{\dot{g}} \dot{g}) = -\tilde{P}\left(\widetilde{\text{grad }} \tilde{V}(g(t))\right).
    \end{equation*}
    Using the properties of $\tilde{\nabla}$, one deduces that
    \begin{equation*}
    \tilde{P}(T_{e}L_{g}(\dot{\xi} + \tilde{\nabla}^{\mathfrak{g}}_{\xi}\xi)) = -\tilde{P}\left(\widetilde{\text{grad }} \tilde{V}(g(t))\right).
    \end{equation*}
    Now, since $\tilde{P}\circ (T_{e}L_{g})=(T_{e}L_{g})\circ \mathfrak{P}$ we also have that
    \begin{equation*}
    (T_{e}L_{g}\circ \mathfrak{P})(\dot{\xi} + \tilde{\nabla}^{\mathfrak{g}}_{\xi}\xi) = -\tilde{P}\left(\widetilde{\text{grad }} \tilde{V}(g(t))\right).
    \end{equation*}
    Equivalently, applying $T_{g}L_{g^{-1}}$ to both sides and using again the expression relating the projections
    \begin{equation*}
        \mathfrak{P}(\dot{\xi} + \tilde{\nabla}^{\mathfrak{g}}_{\xi}\xi) = -\mathfrak{P} \circ T_{g}L_{g^{-1}}\left(\widetilde{\text{grad }} \tilde{V}(g(t))\right).
    \end{equation*}
    Finally, since $\dot{g}\in \tilde{\mathcal{D}}$ we have that $\xi\in\mathfrak{d}$ and $\dot{\xi}\in \mathfrak{d}$. Therefore, using the fact that $\nabla^{\frakd}_\xi \eta = \mathfrak{P}(\nabla_{\xi}^\g\eta)$, for all $\xi, \eta \in \frakd$ we have that
    \begin{equation*}
        \dot{\xi} + \tilde{\nabla}^{\mathfrak{d}}_{\xi}\xi = -\mathfrak{P} \circ T_{g}L_{g^{-1}}\left(\widetilde{\text{grad }} \tilde{V}(g(t))\right).
    \end{equation*}
    Reversing all the arguments, we conclude that (2) also implies (1).
    
    Suppose (1) holds. The curve $g(t)$ is horizontal and by definition $\D=T\pi(\tilde{\mathcal{D}})$ so $g(t)=\tilde{q}(t)$ and $\dot{q}\in \D$. Also, if $g$ is a nonholonomic trajectory, it satisfies \[\tilde{\nabla}^{nh}_{\dot{g}}\dot{g}=-\tilde{P}\left(\widetilde{\text{grad}\tilde{V}}\right).\] It is not difficult to prove a similar result to that from Lemma \ref{nh:trajectory:lift} in the presence of a potential function. Indeed, the previous equation is equivalent to $$\tilde{\nabla}_{\dot{g}}\dot{g}+\widetilde{\text{grad }\tilde{V}}\in \tilde{\mathcal{D}}^{\bot}$$
    and $\dot{g}\in\tilde{\mathcal{D}}$. The fact that the gradient vector field $\widetilde{\text{grad } \tilde{V}}$ projects onto $\text{grad } V$ and from Proposition \ref{pi:related:connections}, we have that
    $$\nabla_{\dot{q}}\dot{q}+\text{grad }V(q(t))\in \D^{\bot} \text{ and } \dot{q}\in \D$$
    which implies that
    \[\nabla^{nh}_{\dot{q}}\dot{q}=-\mathcal{P}\left(\text{grad}V(q(t))\right).\]

    Conversely, if (3) holds then $g(t)$ is not only horizontal but also $\dot{g}(t)\in \tilde{\D}$. Let $h(t)$ be a curve on $G$ satisfying the equation
    \[\tilde{\nabla}^{nh}_{\dot{h}}\dot{h}=-\tilde{P}\left(\widetilde{\text{grad}\tilde{V}}\right),\]
    with intial condition $h(0)= g(0)$ and $\dot{h}(0)=\dot{g}(0)\in \tilde{\D}$. $h(t)$ is a curve whose velocity lies in $\tilde{\D}$ for all $t$ and projects to the unique solution $p(t)$ of the equation
    \[\nabla^{nh}_{\dot{p}}\dot{p}=-\mathcal{P}\left(\text{grad}V(p(t))\right),\]
    with initial position $p(0)=q(0)$ and initial velocity $\dot{p}(0)=\dot{q}(0)$. Hence, $p(t)=q(t)$, and by uniqueness of the horizontal lift $h(t)=g(t)$.
    \end{proof}

\section{Virtual nonholomic constraints on Riemannian homogeneous spaces}\label{sec4}

%\subsection{Virtual nonholonomic constraints}

%\subsection{Holonomic mechanical systems}

\subsection{Virtual constraints}

Next, we briefly recall  the concept of virtual  constraints on a $n$-dimensional manifold $Q$. Suppose that we have a mechanical controlled system, where the control force $F:TQ\times U \rightarrow T^{*}Q$ is of the form
\begin{equation}\label{force:control}
    F(q,\dot{q},u) = \sum_{a=1}^{m} u^{a}\alpha^{a}(q)
\end{equation}
where $\alpha^{a}$ is a one-form on $Q$  with $m<n$, $U\subset\mathbb{R}^{m}$ the set of controls and $u^a\in\mathbb{R}$ with $1\leq a\leq m$ the control inputs. Then, the Riemannian form of the equations of motion reads
\begin{equation}\label{mechanical:control:system}
    \nabla_{\dot{q}(t)} \dot{q}(t) = \text{grad } V(q(t)) + u^{a}(t)Y_{a}(q(t)),
\end{equation}
with $Y_{a}=\sharp(\alpha^{a})$ the corresponding force vector fields. The distribution $\mathcal{F}\subseteq TQ$ generated by the vector fields  $Y_{a}$ is called the \textit{input distribution} associated with the mechanical control system \eqref{mechanical:control:system}. 

\vspace{.2cm}

\begin{definition}
A \textit{virtual constraint} associated with the mechanical control system \eqref{mechanical:control:system} is a controlled invariant distribution $\mathcal{D}\subseteq TQ$ for that system, that is, %$\mathcal{D}\subseteq TQ$ is said to be controlled invariant for the controlled system \eqref{lagrangian:control:system} if
there exists a control function $\hat{u}:\mathcal{D}\rightarrow \mathbb{R}^{m}$ such that the solution of the closed-loop system satisfies $\phi_{t}(\mathcal{D})\subseteq \mathcal{D}$, where $\phi_{t}:TQ\rightarrow TQ$ denotes its flow.
\end{definition}

Provided that $\mathcal{F}$ and $\mathcal{D}$ are complementary distributions, there is a unique control law for the system \eqref{mechanical:control:system} making $\mathcal{D}$ a virtual constraint, see \cite{Consol:Costal:Maggiore}, \cite{Consol:Constal} for the holonomic case and \cite{Simoes:linear:nonholonomic}, \cite{stratoglou2023bvirtual}, \cite{stratoglou2023geometry}, for the nonholonomic one. %\textcolor{red}{Cite more previous papers nonholonomic constraints}

\subsection{Virtual nonholonomic constraints on Riemannian homogeneous spaces}

Suppose that $H$ is a homogenous manifold, acted by a Lie group $G$, $\pi:G \rightarrow H$ denotes the associated projection and consider a mechanical control system of the type \eqref{mechanical:control:system} on the homogeneous manifold $H$. Suppose that the force vector fields $Y_{a}$ are $G$-invariant. In particular, the input distribution $\mathcal{F}$ is $G$-invariant and $\mathcal{F}_{q} = (T_{\pi(e)}\Phi_{g}) ( \mathcal{F}_{\pi(e)} )$, for $q\in H$ and $g\in G$ such that $q=\Phi_{g}(\pi(e))$. In addition, suppose that $\mathcal{D}$ is also a $G$-invariant distribution on $H$ so that $\mathcal{D}_{q}=(T_{\pi(e)}\Phi_{g})  (\mathcal{D}_{\pi(e)})$, where $q\in H$.

%Consider a $G$-invariant control force $F:TH \times U \rightarrow T^{*}H$ of the type \eqref{force:control} with $m<k<n$ where $n=\dim G$, $k=\text{rank } H$ and $U\subset\mathbb{R}^{m}$ the set of controls. Therefore there exists a function $f:T_{\pi(e)}H \times U \rightarrow T_{\pi(e)}^{*}H$ such that $$F(v_q, u) = g \cdot f(T_{q}\Phi_{g^{-1}}(v_{q}),u),$$ where $q\in H, \ v_{q}\in T_{q}H, \ u\in U$ and $g\in G$ is such that $q=\Phi_{g}(\pi(e))$.

Using the identification between $\h$ and $T_{\pi(e)}H$ described in the previous section, there exists a subspace $\mathfrak{d}$ of $\mathfrak{h}$ such that $T_{e}\pi (\mathfrak{d}) = \mathcal{D}_{\pi(e)}$. Likewise, there exist a subspace $\mathfrak{f}\subseteq \h$ such that $T_{e}\pi (\mathfrak{f}) = \mathcal{F}_{\pi(e)}$. These identifications are particularly important for reproducing the trajectories of the Lie algebra on the homogeneous space and vice-versa. 

{On the Lie algebra $\g$, consider the controlled mechanical system of the form
\begin{equation}\label{control:system}
    \dot{g} = T_eL_g\xi, \quad \dot{\xi} + \tilde{\nabla}_{\xi}^\g \xi + T_g L_{g^{-1}}\left(\widetilde{\text{grad}} \ \tilde{V} (g(t)) \right) = \tilde{u}^{a}f_{a},
\end{equation}
where the vectors $f_{a}\in \h$ are defined by $T_{e}\pi(f_{a})=Y_{a}(\pi(e))$ and span the control input subspace
$\mathfrak{f} = \text{span} \{f_{1}, \dots, f_{m}\}.$ It is not difficult to prove that this system evolves inside the horizontal bundle, provided that $\xi$ starts in $\h$. However, we are interested in something else. We would like to know if there exists a control law, that is a function $\tilde{u}:\frakd \to U$, forcing the trajectories to remain in $\frakd$.}

\vspace{.1cm}

\begin{definition}
  The above subspace $\mathfrak{f}$ of the horizontal space $\h \subseteq \g$  is called the  \textit{control input subspace} associated with the mechanical control system \eqref{control:system}. 
\end{definition}

\vspace{.1cm}

\begin{definition}
   A virtual nonholonomic constraint associated with the mechanical system of type \eqref{control:system} is a controlled invariant subspace $\mathfrak{d}$ of $\h$, that is, there exists a control law  making the subspace $\mathfrak{d}$ invariant under the flow of the closed-loop system, i.e. $\xi(0)\in\mathfrak{d}$ and $\xi(t)\in\mathfrak{d}, \; \forall t\geq 0$.
\end{definition}

With this definition we can state the main result of this section: the existence of a control law making $\frakd$ a virtual nonholonomic constraint.

\vspace{.1cm}

\begin{theorem}\label{maintheorem}
    Suppose $\h = \mathfrak{f}\oplus \mathfrak{d} $. Then there exists a unique control law $\tilde{u}^{*}$ making $\mathfrak{d}$ a virtual nonholonomic constraint for the controlled mechanical system \eqref{control:system}.
\end{theorem}

\begin{proof}
    Let $\dim \mathfrak{d} = d$ and $\dim \mathfrak{f}=m=k-d$. Consider the covectors $\mu^{1}\dots, \mu_{m} \in \h^{*}$ spanning the annihilator subspace of $\mathfrak{d}$. $\xi(t)$ is a curve on $\h$ satisfying $\xi(t)\in \mathfrak{d}$ for all time if and only if it satifies $\mu^{a}(\xi(t))= 0$ for all $a=1,\dots, m$. Differentiating this equation and supposing that $\xi(t)$ is a solution of the closed loop system \eqref{control:system} for an appropriate choice of control law $\tilde{u}$, we have that
    $$-\mu^{a}\left( \tilde{\nabla}_{\xi}^\g \xi + (T_{g} L_{g^{-1}})\left(\widetilde{\text{grad}} \ \tilde{V} (g(t)) \right)\right) + \tilde{u}^{b}\mu^{a}(f_{b}) = 0.$$
    Since $\tilde{\nabla}_{\xi}^\g \xi + (T_{g} L_{g^{-1}})\left(\widetilde{\text{grad}} \ \tilde{V} (g(t)) \right) \in \h$ and $\h = \mathfrak{f}\oplus \mathfrak{d} $, there is a unique way to decompose this vector as the sum
    $$\tilde{\nabla}_{\xi}^\g \xi + (T_{g} L_{g^{-1}})\left(\widetilde{\text{grad}} \ \tilde{V} (g(t)) \right) = \eta(t) + \tau^{b}(t)f_{b},$$
    with $\eta(t) \in \mathfrak{d}$. In addition, note that the coefficients $\tau^{b}$ may be regarded as functions on $\h$. In fact, its definition is associated with the projection to $\mathfrak{f}$ together with the choice of $\{f_{b}\}$ as a basis for $\mathfrak{f}$. Therefore, $\mu^{a}(\dot{\xi}(t))=0$ if and only if
    $$(\tau^{b}-\tilde{u}^{b}) \mu^{a}(f_{b})=0.$$
    Since $\mu^{a}(f_{b})$ is an invertible matrix, we conclude that $\tau^{b}=\tilde{u}^{b}$ proving existence and uniqueness of a control law making $\mathfrak{d}$ a virtual constraint.
\end{proof}

\begin{remark}
    The transversality assumption appearing in the previous theorem is related to another appearing in the literature of virtual holonomic constraints (see \cite{Consol:Constal}, \cite{Consol:Costal:Maggiore}). In particular, the assumption of (vector) relative degree $\{1,\ldots,1\}$ appearing in the literature of zero dynamics manifolds (see \cite{isidori1985nonlinear}) concerning control systems evolving in Euclidean spaces. If $\mathcal{F}_{g}=T_{e}L_{g}(\mathfrak{f})$ is the left invariant distribution induced by $\mathfrak{f}$, $\tilde{\mathcal{D}}_{g}=T_{e}L_{g}(\mathfrak{d})$ is the left invariant distribution induced by $\mathfrak{d}$, $\phi:TQ\to \R^{m}$ a map annihilating on $\tilde{\mathcal{D}}$ and $Y^{a}\in \mathcal{F}$ are vectors spanning $\mathcal{F}$, then the relative degree of $\phi$ is $\{1,\ldots,1\}$ if $\langle d\phi(v_{q}), (Y^{a})_{v_{q}}^{V} \rangle \neq 0$ for all $a$, which holds by virtue of the fact that $Y$ is not contained in $\tilde{\mathcal{D}}$.
\end{remark}

By construction, if $(g(t), \xi(t))$ is a trajectory of the controlled mechanical system \eqref{control:system}, then $g(t)$ is a horizontal curve. We will prove next that $g$ is in fact the horizontal lift of the curve $q(t)=\pi(g(t))$ and that $q$ is the trajectory of a controlled mechanical system of the form
$$\nabla_{\dot{q}(t)} \dot{q}(t) = - \text{grad} \ V (q(t)) + u^{a}(t)Y_{a}(q(t)).$$

In addition, if $\xi(t) \in \mathfrak{d}$ for all $t$, then $\dot{q}(t)\in \mathcal{D}_{q(t)}$ for all $t$. Then, by uniqueness of the control law making the distribution $\mathcal{D}$ a virtual nonholonomic constraint, the control law $\tilde{u}^{*}\in U$ given in Theorem \ref{maintheorem} must also be the unique control law making $\mathcal{D}$ control invariant. We summarize these facts in the next result.

\vspace{.2cm}

{\begin{theorem}\label{homogeneous:virtual:constraint:thm}
    Let $\tilde{u}^{*}:\frakd\to U$ be the unique control law given by Theorem \ref{maintheorem} and $(g,\xi):[a,b]\to G\times\g$ the solution of the associated closed-loop system
    \begin{equation}\label{virtual:closed:loop:system}
        \begin{split}
            \dot{g} & = (T_{e} L_{g}) (\xi) \\
            \dot{\xi} & + \tilde{\nabla}_{\xi}^\g \xi + (T_{g} L_{g^{-1}})\left(\widetilde{\text{grad}} \ \tilde{V} (g(t)) \right) = \tilde{u}^{* a}(\xi)f_{a},
        \end{split}
\end{equation}
with $\xi(0)\in \frakd$. If $q(t)=\pi(g(t))$ then $\dot{q}(t)\in \D_{q(t)}$, $g(t)=\tilde{q}(t)$ and $q$ satisfies
$$\nabla_{\dot{q}(t)} \dot{q}(t) = - \text{grad} \ V (q(t)) + u^{* a}(t)Y_{a}(q(t)).$$
with $u^{* a}(t):= u^{* a}(q(t),\dot{q}(t))=\tilde{u}^{* a}(\xi(t))$.
\end{theorem}}

\begin{proof}
    By construction of the $\g$-connection, if $(g,\xi)$ satisfies equation \eqref{control:system}, then $g$ satisfies
    $$\tilde{\nabla}_{\dot{g}}\dot{g} + \widetilde{\text{grad}} \ \tilde{V} (g(t)) = \tilde{u}^{* a}((T_{g} L_{g^{-1}})(\dot{g})) (f^{L})_{a},$$
    where $f^{L}_{a}=T_{e}L_{g}(f_{a})$. Moreover, since $\xi(t)\in \frakd$ in $[a,b]$ we have that $\dot{g}(t)\in \tilde{\mathcal{D}}$. In particular, $g$ is horizontal and, hence, $g(t)=\tilde{q}(t)$.

    {By Proposition \ref{pi:related:connections}, $T\pi$ projects $\tilde{\nabla}_{\dot{\tilde{q}}}\dot{\tilde{q}}$ to $\nabla_{\dot{q}} \dot{q}$. In addition, from Lemma \ref{grad:horizontal}, $\widetilde{\text{grad}} \ \tilde{V} (\tilde{q}(t))$ is the horizontal vector field projected onto $\text{grad }V(q(t))$ by $T\pi$. Finally, $$T_{g}\pi (f^{L}_{a}(g))= T_{g}\pi (T_{e}L_{g}(f_{a})) = T_{\pi(e)}\Phi_{g} (T_{e} \pi (f_{a})),$$
    where the last equality comes from the relation $\pi \circ L_{g} = \Phi_{g} \circ \pi$. Now, using that $Y_{a}(e) = T_{e}\pi(f_{a})$ and the $G$-invariance of $Y_{a}$, implying that $Y_{a}(g)=T_{\pi(e)}\Phi_{g}(Y_{a}(e))$ we deduce that $T_{g}\pi (f^{L}_{a}(g))=Y_{a}$.
    
    If we define $u^{* a}(q(t),\dot{q}(t)):=\tilde{u}^{ *a}(T_{\tilde{q}(t)}L_{\tilde{q}(t)^{-1}}\dot{\tilde{q}}(t)) $, we obtain the desired result.}\end{proof}

\subsection{The induced constrained connection}

Let us define the projections $\mathfrak{p}_{\frakd}:\g\to\frakd$ and $\mathfrak{p}_{\mathfrak{f}}:\g\to \mathfrak{f} \oplus \mathfrak{s}$, associated with the direct sum $\h = \frakd \oplus \mathfrak{f}$. Note that, from the first decomposition of the Lie algebra in terms of horizontal and vertical spaces, i.e., $\g = \h \oplus \mathfrak{s}$, we obtain the three-part decomposition of the Lie algebra $\g = \frakd \oplus (\mathfrak{f} \oplus \mathfrak{s})$ to which the above projections are associated.

Now let us define the Lie algebra connection called $\tilde{\nabla}^{\frakd,\mathfrak{f}}$-connection given by
$$\tilde{\nabla}^{\frakd,\mathfrak{f}}_{\xi}\eta = \tilde{\nabla}^{\g}_{\xi} \eta + (\tilde{\nabla}^{\g}_{\xi}\mathfrak{p}_{\mathfrak{f}})(\eta),$$

{\begin{theorem}\label{intrinsic:virtual:constraint:thm}
    The trajectory of the closed-loop system \eqref{virtual:closed:loop:system} is a trajectory of the equations
    \begin{equation}\label{intrinsic:equations}
        \begin{split}
            \dot{g} & = (T_{e} L_{g}) (\xi), \\
            \dot{\xi} & + \tilde{\nabla}_{\xi}^{\frakd,\mathfrak{f}} \xi + \mathfrak{p}_{\frakd}\left((T_{g} L_{g^{-1}})\left(\widetilde{\text{grad}} \ \tilde{V} (g(t)) \right)\right) = 0.
        \end{split}
\end{equation}
\end{theorem}}

\begin{proof}
     It is not difficult to prove that if $\xi \in \frakd$, then $\tilde{\nabla}_{\xi}^{\frakd,\mathfrak{f}} \xi = \mathfrak{p}_{\frakd} \left( \tilde{\nabla}^{\g}_{\xi} \xi \right)$. Attending to the fact that $\xi(t)\in \frakd$ along the solutions of the closed loop system, then we have that
     $$\dot{\xi} + \tilde{\nabla}_{\xi}^{\frakd,\mathfrak{f}} \xi = \dot{\xi} + \mathfrak{p}_{\frakd} \left( \tilde{\nabla}^{\g}_{\xi} \xi \right)= \mathfrak{p}_{\frakd} \left( \dot{\xi} + \tilde{\nabla}^{\g}_{\xi} \xi \right)$$
     And using that $\mathfrak{p}_{\frakd}(f_{a}) = 0$, we deduce
     $$\dot{\xi} + \tilde{\nabla}_{\xi}^{\frakd,\mathfrak{f}} \xi = -\mathfrak{p}_{\frakd}\left((T_{g} L_{g^{-1}})\left(\widetilde{\text{grad}} \ \tilde{V} (g(t)) \right)\right)$$
\end{proof}

\begin{remark}
    Consider the affine connection $\nabla^{\D,\mathcal{F}}$ on $H$ associated with the projections $P_{\mathcal{D}}:TH\to\D$ and $P_{\mathcal{F}}:TH\to \mathcal{F}$. Then the trajectories $(g,\xi)$ of \eqref{intrinsic:equations} generate a trajectory $q(t)=\pi(g(t))$ which satisfies
    $$\nabla^{\D,\mathcal{F}}_{\dot{q}(t)} \dot{q}(t) = - P_{\mathcal{D}}(\text{grad} \ V (q(t))).$$
    This is a consequence of Theorems \ref{homogeneous:virtual:constraint:thm} and \ref{intrinsic:virtual:constraint:thm} together with Theorem 2 in \cite{Simoes:linear:nonholonomic}.
\end{remark}

%%%%%%%%%%%%%%%%%%%%%%%%%%%%%%%%%%%%%%%%%%%%%%%%%%%%%%%%%%

%\section{Examples}
\subsection{Example: A sphere rolling over another sphere}\label{ex1}

\vspace{-.2cm}

Consider a sphere rolling on another sphere. For simplicity we consider the radius of the rolling sphere to be 1 and that of the stationary sphere equal to $\rho>1.$ The rolling sphere is equipped with an orthonormal frame fixed at its center. The motion of the sphere is completely described by its position on the standing sphere given by a normalized vector $r\in\Sp^2$ and by a rotation matrix $R\in\SO(3)$ which gives the orientation of the frame of the rolling sphere related to a fixed frame at the center of the stationary one. Thus, the configuration space is $H=\Sp^2\times\SO(3).$

\vspace{-.2cm}
%\textit{Homogeneous structure of $\Sp^2\times\SO(3)$.}
The Lie group $\SO(3)$ acts transitively on the 2-sphere $\Sp^2$ by left multiplication, hence the latter is a homogeneous space i.e. it can be seen as the quotient of the $\SO(3)$ and a closed subgroup of $\SO(3)$, namely $\Sp^2\simeq\SO(3)/\tilde{K}$, where the subgroup $\tilde{K}$ is isomorphic to $\SO(2)$ expressing its elements $k\in \tilde{K}$ by $k=\begin{pmatrix}
    \cos(\theta) & -\sin(\theta) & 0 \\
    \sin(\theta) & \cos(\theta) & 0 \\
    0 & 0 & 1
\end{pmatrix} \text{ with } \theta\in \mathbb{S}^1$. For the second component of the configuration manifold $H$, that is $\SO(3)$, which is a Lie group, the same occurs trivially with the left-multiplication on itself and the homogeneous structure is $\SO(3)\simeq\SO(3)/I$, where $I$ is the identity matrix. Ultimately, the Lie group $G=\SO(3)\times\SO(3)$ acts transitively on $H$ by the Lie group action $\Phi:G\times H\to H$, with $\Phi_{(S,R)}(r,T)=(Sr,RT).$ For $e_3=[0 \; 0 \; 1]^T$, we define $K=\text{Stab}((e_3,I))$  which is the stabilizer subgroup of the action $\Phi$. The projection map $\pi:G\to H$ is given by $\pi(S,R)=(Se_3,RI)$. We will denote the projection $SO(3)\rightarrow \mathbb{S}^{2}$ also by $\tilde{\pi}$, i.e., $\tilde{\pi}(S)=Se_3, S\in \SO(3).$

\vspace{-.2cm}
%\textit{The Lie algebra $\
%so(3)\times \so(3)$.}
Using the hat map, we identify $\so(3)\cong \mathbb{R}^3$. Consider the orthonormal basis $\{\hat{e}_1, \hat{e}_2, \hat{e}_3\}$ of $\so(3)$, where $e_1, e_2, e_3$ is the standard basis on $\mathbb{R}^3$ so, we have $\pi_*(\hat{e}_1,\cdot)=(-e_2,\cdot), \pi_*(\hat{e}_2,\cdot)=(e_1,\cdot)$ and $\pi_*(\hat{e}_3,\cdot)=(0, \cdot)$.

\vspace{-.2cm}

%\textit{Riemannian homogeneous structure.} 
Suppose the Lie group $G=\SO(3)\times\SO(3)$ is equipped with a left-invariant metric given by the inner product on $\g=\so(3)\times\so(3)$, i.e. $\langle (\hat{\Pi}_1,\hat{\Omega}_1), (\hat{\Pi}_2,\hat{\Omega}_2)\rangle_{\g}=\langle\hat{\Pi}_1,\hat{\Pi}_2\rangle_{\so(3)} + \langle\hat{\Omega}_1,\hat{\Omega}_2\rangle_{\so(3)} = \Pi^T_1\Pi_2 + \Omega_1^T\mathbb{J}\Omega_2$ for all $\Pi_1,\Pi_2,\Omega_1,\Omega_2\in \mathbb{R}^3$ and $\mathbb{J}$ the moment of inertia tensor. Using this left-invariant metric of the Lie group $G$ we define an inner product on $T_{e_3}\mathbb{S}^2\times\so(3)$ via the relation $\langle X, Y \rangle_{T_{e_3}\mathbb{S}^2\times\so(3)}:=\langle\pi^{-1}_{*}X,\pi^{-1}_*Y\rangle_{\g}$ $= x\cdot y + \bar{X}^T\mathbb{J} \bar{Y}$ for all $X, Y \in T_{e_3}\mathbb{S}^2\times\so(3)$ where $X=(x,\bar{X})$ and $Y=(y,\bar{Y})$. Following \cite{goodman2024reduction} we have that the first part of $\langle\cdot,\cdot\rangle_{T_{e_3}\mathbb{S}^2\times\so(3)}$ is the standard Euclidean metric with respect to the basis $\{e_1,e_2\}$. Thus, we can extend this inner product to an SO(3)-invariant Riemannian metric on $H=\Sp^2\times\SO(3)$ by left-action given by
$\langle X,Y\rangle_{H}=\langle \bar{R}^{-1}X,\bar{R}^{-1} Y\rangle_{T_{e_3}\mathbb{S}^2\times\so(3)}=\langle\pi^{-1}_{*}(\bar{R}^{-1} X),\pi^{-1}_*(\bar{R}^{-1} Y)\rangle_{\g}= S^{T}x \cdot S^{T}y + (R^{T}\bar{X})^{T}\mathbb{J} R^{T}\bar{Y}=x \cdot y + \bar{X}^{T} (R\mathbb{J}R^{T})\bar{Y}$.
for all $X,Y\in T_q\mathbb{S}^2\times T_R\SO(3),\; \bar{R}=(S,R)\in G$ such that $\pi(S,R)=(q,R),$ where $\bar{R}^{-1}X=(S^{-1},R^{-1})\cdot(x,\bar{X})=(S^{-1}x,R^{-1}\bar{X})$ for $X=(x,\bar{X})$ and $Y=(y,\bar{Y})$.

%\textit{The subspace $\mathfrak{h} \subseteq \so(3)\times \so(3)$.} 
With the Riemannian homogeneous structure above, we have that $\s=\ker(\pi_*|_{\g})=\text{span}\{(\hat{e}_3,0)\}$ and we define $\mathfrak{h}=\mathfrak{s}^\perp$ such that $\mathfrak{h}=\text{span}\{(\hat{e}_1,0), (\hat{e}_2,0), (0, \hat{e}_1), (0, \hat{e}_2), (0, \hat{e}_3) \}$. For the inner product on $\mathfrak{g}$, the flat map $\flat_{\g}:\g \to \g^{*}$ is given by $\flat_{\g}(\Pi, \Omega)=(\Pi, \mathbb{J}\Omega)$ and its inverse, the sharp map $\sharp_{\g}: \g^*\to \g$, is given by $\sharp_{\g}{(\mu, \nu)}=(\mu,\mathbb{J}^{-1}\nu)$, where $\mu$ and $\nu$ are vectors in $\R^{3}$ identified with the matrices $\bar{\mu}$ and $\bar{\nu}$ in $\g^{*}$ through the dual pairing $\left\langle \bar{\mu}, \hat{\Pi}\right\rangle = \mu^{T}\Pi$ and $\left\langle \bar{\nu}, \hat{\Omega}\right\rangle = \nu^{T}\Omega$.
The adjoint operator $\ad:\g\times\g\to\g$ for $\g=\so(3)\times \so(3)$ is given by  
\begin{align*}
    \ad_{\xi}\eta&=(\ad_{\hat{\Pi}_1}\hat{\Pi}_2,\ad_{\hat{\Omega}_1}\hat{\Omega}_2)=([\hat{\Pi}_1,\hat{\Pi}_2],[\hat{\Omega}_1,\hat{\Omega}_2]) \\
    &=(\widehat{\Pi_1\times\Pi_2},\widehat{\Omega_1\times\Omega_2})
\end{align*}
where $\ad_{\hat{\Pi}_1}\hat{\Pi}_2$ and $\ad_{\hat{\Omega}_1}\hat{\Omega}_2$ are the adjoint operators on $\so(3)$ given by the cross product of vectors on $\R^3$ using the hat map, $\xi=(\hat{\Pi}_1, \hat{\Omega}_{1})$ and $\eta=(\hat{\Pi}_2, \hat{\Omega}_{2})$. The co-adjoint operator is given by $\ad^{*}_{(\hat{\Pi}, \Omega))} (\mu, \nu) = (\mu \times \Pi, \nu \times \Omega)$.

Since $\mathfrak{s}=\text{span}\{(\hat{e}_3,0)\}$, the vertical space of $G$ at $g=(S,R)$ is defined by $\text{Ver}_S\times \{0\}=\text{span}\{(T_{I}L_{S}(\hat{e}_3),0)\}$ and the horizontal space is defined as $\text{Hor}_S\times T_{R}\SO(3)$ where $\text{Hor}_S=\text{span}\{(T_{I}L_{S}(\hat{e}_1),T_{I}L_{S}(\hat{e}_2) )\}$. The horizontal projection can be calculated by $\mathcal{H}(\hat{\Pi},\hat{\Omega})=(\widehat{\Pi\times e_3},\hat{\Omega}).$

%\textcolor{blue}{if $\flat(\Omega)=\Omega^T\mathbb{J}$ why at the second entry below we have $(\mathbb{J}\Omega_1\times\Omega_2)$ and not $(\Omega_1^T\mathbb{J}\times\Omega_2)$? shall we follow \cite{goodman2024reduction} 2.2 Example 1? i.e. $\flat(\hat{\Pi},\hat{\Omega})(\hat{\xi},\hat{\eta})=\flat(\hat{\Pi})\hat{\xi} + \flat(\hat{\Omega})\hat{\eta} = \langle\hat{\Pi},\hat{\eta}\rangle + \langle\hat{\Omega},\hat{\eta}\rangle = \Pi^T\xi + \Omega^T\mathbb{J}\eta = \Pi^T\xi + (\mathbb{J}\Omega)^T\eta$ so we have $\flat(\hat{\Pi},\hat{\Omega})=(\Pi,\mathbb{J}\Omega)$.} 

The $\mathfrak{g}$-connection is given by
\begin{equation*}
    \tilde{\nabla}^\g_\xi\eta  =\frac{1}{2} \left(\widehat{(\Pi_{1} \times \Pi_{2})}, \widehat{\Omega_1\times\Omega_2} -\widehat{\mathbb{J}^{-1}(\mathbb{J}\Omega_1\times\Omega_2 - \mathbb{J}\Omega_2\times\Omega_1)} \right),  
\end{equation*} where $\xi=(\hat{\Pi}_1,\hat{\Omega}_1)$ and $\eta=(\hat{\Pi}_2, \hat{\Omega}_2).$

For a horizontal curve $\bar{R}:[a,b]\to G$ we have from Lemma \ref{lemma2} %\ref{lemma: cov-to-covg}
that
\begin{align*}
    \tilde{\nabla}_{\dot{\bar{R}}} \dot{\bar{R}}(t) &= \bar{R}(t)\left(\dot{\xi} + \tilde{\nabla}_{\xi}^\g \xi(t) \right)  \\
    &= (S,R) \left(\dot{\hat{\Pi}}, \dot{\hat{\Omega}} -\left(\widehat{\mathbb{J}^{-1}(\mathbb{J}\Omega\times\Omega)}\right) \right),
\end{align*} where $\xi=\bar{R}^{-1}\dot{\bar{R}}$ and $\bar{R}=(S,R)$. In particular, if $\bar{R}$ is a horizontal geodesic then from the respective equation %\eqref{LP: geo}
$\xi = (\hat{\Pi}, \hat{\Omega})$ satisfies
\begin{equation*}
         \dot{\Pi} = 0, \quad
         \mathbb{J}\dot{\Omega} = \mathbb{J}\Omega\times\Omega.
\end{equation*}
Note here that the second equation is the usual Euler equation for a rigid body.

%\textcolor{magenta}{Tony:Not sure this matters, but technically shouldn't the rigid body left hand side be $\mathbb{J}\dot{\Omega}$? I really like the computations in this example. } 

%\textit{A distribution on $H$.}
Suppose we want to impose the non-slipping condition expressed by the nonholonomic constraints equations 
\begin{equation*}
         \dot{q}\cdot Re_{1} = - e_{2}^{T}\omega, \hbox{ and }
         \dot{q}\cdot Re_{2} = e_{1}^{T}\omega
\end{equation*}
which, in the north pole, can be written as 
\begin{equation*}
        \dot{x} = - e_{2}^{T}\omega, \quad
         \dot{y} = e_{1}^{T}\omega.
\end{equation*}
The constraint in the tangent space of the north pole, $T_{q}H$, can be written as $\text{span}\{(e_{2},\hat{e}_1), (-e_{1},\hat{e}_2), (0,\hat{e}_3)\}\subseteq T_{q}H$ and the same constraints expressed in the Lie algebra $\so(3)$ define the subspace $\mathfrak{d}=\text{span}\{(-\hat{e}_{1},\hat{e}_1), (-\hat{e}_{2},\hat{e}_2), \\ (0,\hat{e}_3)\}\subseteq\mathfrak{h}$. Thus, we define $\mathfrak{f}=\text{span}\{(\hat{e}_{1},\hat{e}_1),(\hat{e}_{2},\hat{e}_2)\}$ and we look for a control law that makes the system 
\begin{equation} \label{closed loop system}
     \dot{\Pi} = \textbf{u}, \quad
        \mathbb{J}\dot{\Omega} = \mathbb{J}\Omega\times\Omega + \textbf{u}
\end{equation}
control invariant, where $\textbf{u}=(u_1,u_2,0)\in \R^3$.

%\textit{The control law u.} 
Let us consider $\mathbb{J}$ to be a diagonal matrix with entries 
$J_i, i=1,2,3.$ Differentiating the constraint equations and expressing them in terms of the Lie group, we get
\begin{equation*}
        \dot{\Pi}_2 = - e_{2}^{T}\dot{\Omega}, \quad
        \dot{\Pi}_1 = - e_{1}^{T}\dot{\Omega}.
\end{equation*}
Thus, using the equations (\ref{closed loop system}), we have that the unique control law that makes $\mathfrak{d}$ a virtual nonholonomic constraint is 
\begin{equation*}
         u_1 = \frac{J_3-J_2}{J_1+1}\Omega_2\Omega_3, \quad
        u_2 = \frac{J_1-J_3}{J_2+1}\Omega_1\Omega_3.
\end{equation*}

\subsection{Example: Blade moving on a sphere}\label{ex2}
%A Chaplygin sleigh rolling over a sphere
%Consider the Chaplygin sleigh  but instead of moving on a plane as usual it is moving on a sphere. The Chaplygin sleigh is a rigid body that is supported at three points, two that can slide without friction and one knife edge.
Consider a blade moving on a sphere. To analyse the system fix a great circle on the sphere, called the equator and a coordinate system fixed on the body $\{e_{1}, e_{2}\}$, attached to the point of contact of the blade. The configuration of the body is described by its position on the sphere $r\in\Sp^2$, and the angle $\vartheta$ defined as the angle between the tangent vector to the equator and the velocity vector of the geodesic passing through $r$ with direction $e_{1}$ at the point at which the two great circles intersect. Hence, the configuration is $H=\Sp^2\times\Sp^1$. %\textcolor{blue}{if there is space we can add a figure}

\vspace{-.2cm}

%\textit{Homogeneous structure of $\Sp^2\times\Sp^1$.} 
Regarding the first component of $H$, the analysis in the previous example applies here as well and $\Sp^1$ is a Lie group so $H$ is a homogeneous space. Thus, we have that the Lie group $G=\SO(3)\times\Sp^1$ acts transitively on $H$ by the action $\Psi:G\times H\to H,$ with $\Psi_{(S,\varphi)}(r,\vartheta)=(Sr,\varphi+\vartheta)$. For $e_3=[0 \; 0 \; 1]^T$ we define $K=\text{Stab}((e_3,0))$ which is the stabilizer subgroup of the action $\Phi$. The projection map is $\pi:G\to H$ is given by $\pi(S,\varphi)=(Se_3,\varphi)$. By abuse of notation, we will denote the projection $SO(3)\rightarrow \mathbb{S}^{2}$ also by $\pi$, i.e., $\pi(S)=Se_3, S\in \SO(3).$ The Lie algebra is $\so(3)\times\R$, where the first component is as in Example \ref{ex1}.

\vspace{-.2cm}

%\textit{Riemannian homogeneous structure.} 
Suppose the Lie group $G=\SO(3)\times\Sp^1$ is equipped with a left-invariant metric given by the inner product on $\g=\so(3)\times\R$, i.e. $\langle (\hat{\Pi}_1,\omega_1), (\hat{\Pi}_2,\omega_2)\rangle_{\g}=\langle\hat{\Pi}_1,\hat{\Pi}_2\rangle_{\so(3)} + \langle \omega_1,\omega_2\rangle_{\R} = \Pi^T_1\Pi_2 + \omega_1\omega_2$ for all $\Pi_1,\Pi_2\in \mathbb{R}^3$ and $\omega_1,\omega_2\in\R$. Using this left-invariant metric of the Lie group $G$ we define an inner product on $T_{e_3}\mathbb{S}^2\times\R$ via the relation $\langle X, Y \rangle_{T_{e_3}\mathbb{S}^2\times\R}:=\langle\pi^{-1}_{*}X,\pi^{-1}_*Y\rangle_{\g}=x\cdot y + \omega_1\omega_2$  for all $X, Y \in T_{e_3}\mathbb{S}^2\times\R,$ where $X=(x,\omega_1), Y=(y,\omega_2)$. As previously, we can extend this inner product to a $G$-invariant Riemannian metric on $H=\Sp^2\times\Sp^1$ by left-action given by $\langle X,Y\rangle_{H}=\langle \bar{R}^{-1}X,\bar{R}^{-1}Y\rangle_{T_{e_3}\mathbb{S}^2\times\R}= \langle\pi^{-1}_*(\bar{R}^{-1}X), \pi^{-1}_*(\bar{R}^{-1}Y)\rangle_{\g} = S^Tx\cdot S^Ty + \omega_1\omega_2= x\cdot y + \omega_1\omega_2$ for all $X,Y\in T_q\mathbb{S}^2\times T_\theta\Sp^1, \; \bar{R}=(S,\varphi)\in G$ such that $\pi(S,\varphi)=(q,\theta)$ where $\bar{R}^{-1}X=(S^{-1},-\varphi)(x,\omega_1)=(S^{-1}x,\omega_1)$ and $X=(x,\omega_1), Y=(y,\omega_2)$.

\vspace{-.2cm}

%\textit{The subspace $\mathfrak{h} \subseteq \so(3)\times \R$.} 
With the Riemannian homogeneous structure above we have that $\s=\ker(\pi_*|_{\g})=\text{span}\{(\hat{e}_3,0)\}$ and $\mathfrak{h}=\mathfrak{s}^\perp=\text{span}\{(\hat{e}_1,0), (\hat{e}_2,0), (0,1) \}$. Associated with the inner product, the flat map $\flat_{\g}:\g \to \g^{*}$ is given by $\flat_{\g}(\Pi, \omega)=(\Pi, \omega)$ and its inverse, the sharp map $\sharp_{\g}: \g^*\to \g$, is given by $\sharp_{\g}{(\mu, \lambda)}=(\mu,\lambda)$  where $\mu$ is a vector in $\R^{3}$ identified with the matrix $\bar{\mu}$ in $\g^{*}$ through the dual pairing $\left\langle \bar{\mu}, \hat{\Pi}\right\rangle = \mu^{T}\Pi$ and $\lambda\in \R$.
The adjoint operator of $\g=\so(3)\times \R$ to itself is given by $\ad:\g\times\g\to\g$, 
\begin{align*}
    \ad_{\xi}\eta&=(\ad_{\hat{\Pi}_1}\hat{\Pi}_2,\ad_{\omega_1}\omega_2)=([\hat{\Pi}_1,\hat{\Pi}_2],0) =(\widehat{\Pi_1\times\Pi_2},0)
\end{align*}
where $\ad_{\hat{\Pi}_1}\hat{\Pi}_2$ is the adjoint operator on $\so(3)$, $\xi=(\hat{\Pi}_1, \omega_{1})$ and $\eta=(\hat{\Pi}_2, \omega_{2})$. The co-adjoint operator is given by $\ad^{*}_{(\hat{\Pi}, \omega)} (\mu, \nu) = (\mu \times \Pi, 0)$.

\vspace{-.2cm}

Since $\mathfrak{s}=\text{span}\{(\hat{e}_3,0)\}$ the vertical space of $G$ is defined by $\text{Ver}_{S}\times \{0\}=\text{span}\{(T_{I}L_{S}(\hat{e}_3),0)\}$ and the horizontal space is defined as $\text{Hor}_{S}\times\Sp^1$ where $\text{Hor}_{S}=\text{span}\{T_{I}L_{S}(\hat{e}_1),T_{I}L_{S}(\hat{e}_2)\}$ for $S\in \SO(3)$. The horizontal projection in given by $\mathcal{H}(\hat{\Pi},\omega)=(\widehat{\Pi\times e_3},\omega).$

\vspace{-.2cm}

\textit{The $\mathfrak{g}$-connection} is given by
\begin{align*}
    \tilde{\nabla}^\g_\xi\eta&=\frac{1}{2} (\widehat{\Pi_1\times\Pi_2},0),  
\end{align*} where $\xi=(\hat{\Pi}_1,\omega_1)$ and $\eta=(\hat{\Pi}_2, \omega_2).$

\vspace{-.2cm}
For a horizontal curve $\bar{R}:[a,b]\to G$ we have from Lemma \ref{lemma2} %\ref{lemma: cov-to-covg}
that
\[\tilde{\nabla}_{\dot{\bar{R}}} \dot{\bar{R}}(t) = \bar{R}(t)\left(\dot{\xi} + \tilde{\nabla}_{\xi}^\g \xi(t) \right)= (S,R) \left(\dot{\hat{\Pi}}, \dot{\hat{\Omega}}\right),\]

\vspace{-0.8cm} where $\xi=\bar{R}^{-1}\dot{\bar{R}}$ and $\bar{R}=(S,\vartheta)$. In particular, if $\bar{R}$ is a horizontal geodesic then from the respective equation %\eqref{LP: geo}
$\xi = (\hat{\Pi}, \omega)$ satisfies
\begin{equation*}
 \dot{\Pi} = 0,\quad
        \dot{\omega} =0.
\end{equation*}

\vspace{-.2cm}

The above equations in the homogeneous space take the form \begin{equation*}
 R\dot{(\Pi}\times e_3) = 0, \quad
 \dot{\omega} =0.
\end{equation*}

\vspace{-.2cm}

%\textit{A distribution on H.} 
For simplicity consider that the equator passes from the north pole then the constraint on the tangent plane at the north pole is given by the knife edge constraint (see \cite{bloch2003nonholonomic} Section $1.6$) $\dot{x}\sin{\vartheta}=\dot{y}\cos{\vartheta}$.
This equation defines the vector space $\D_{e_{3}} = \text{span}\{X=\cos\vartheta e_1 + \sin\vartheta e_2 \}\times\R\subset T_{e_3}H$ and the same constraints expressed in the Lie algebra $\g$ define the distribution $\mathfrak{d}=\text{span}\{(-\sin\vartheta\hat{e}_{1} +\cos\vartheta\hat{e}_{2},0),(0,1)\}\subseteq\mathfrak{h}$. Thus, we define $\mathfrak{f}=\text{span}\{(\cos\vartheta\hat{e}_{1} + \sin\vartheta\hat{e}_2,0)\}$ and we look for a control law that makes the system 
\begin{equation} \label{closed loop system - knife}
   \dot{\Pi}_1 = u\cos\vartheta,\quad \dot{\Pi}_2 = u\sin\vartheta,\quad\dot{\Pi}_3 = 0,\quad \dot{\omega} = 0
\end{equation}
control invariant.

\vspace{-.2cm}

 Differentiating the constraint equations and expressing them in terms of the Lie group, we get
\begin{equation*}
    \dot{\Pi}_2\sin\vartheta + \dot{\Pi}_2\omega\cos\vartheta +\dot{\Pi}_1\cos\vartheta - \dot{\Pi}_1\omega\sin\vartheta=0
\end{equation*}
Thus, using the equations (\ref{closed loop system - knife}) we have that the unique control law making $\mathfrak{d}$ a virtual nonholonomic constraint is 
\begin{equation*}
    u=\omega\left(\Pi_1\sin\vartheta - \Pi_2\cos\vartheta\right).
\end{equation*}

\section{Conclusions}

We have studied virtual nonholonomic constraints for mechanical control systems evolving on Riemannian homogeneous spaces by examining the dynamics on the Lie algebra of the symmetry group. Taking advantage of the linear structure, we have shown the existence and uniqueness of a control law allowing one to define a virtual nonholonomic constraint and we have characterized the trajectories of the closed-loop system as solutions of a mechanical system associated with an induced constrained connection. %In addition, we have characterized the dynamics of this kind of nonholonomic systems in terms of virtual nonholonomic constraints. 

\vspace{-2mm}
By transporting the control system to a linear space, our methodology could help in designing new control laws to enforce a virtual nonholonomic constraint when the configuration manifold is difficult to tackle directly without the use of symmetries. In addition, the qualitative properties of the closed-loop system are in general easier to investigate in a linear space than in an arbitrary nonlinear configuration manifold. In particular, investigating the stabilization of the virtual constraints, which is part of our future work.

%\section{Example 2}

%Consider a homogeneous ball rolling on a table. The configuration space that describes the motion of the ball is $H=\SO(3)\times\R^2$ which gives the orientation and the position of the ball on the horizontal plane ($z=0$). The first component of $H$ can be identified with the homogeneous space $G/K$ for $G=\SU(2)$ and $K=\{I,-I\}$ where $\SU(2)$ is the special unitary group. The Lie algebra of $\SU(2)$ is $\mathfrak{su}(2)$ is defined by the basis vectors
%\[u_1=\begin{bmatrix}
 %   0 & i \\
  %  i & 0
%\end{bmatrix}, \;
%u_2=\begin{bmatrix}
 %   0 & -1 \\
  %  1 & 0
%\end{bmatrix}, \;
%u_3=\begin{bmatrix}
 %   i & 0 \\
  %  0 & -i
%\end{bmatrix}.\]

%The inner product on $\mathfrak{su}(2)$ is given by $\langle x,y \rangle=Re(\tr(xy))$ for $x,y\in\mathfrak{su}(2)$. The Lie algebras of $\SO(3)$ and $\SU(2)$ are isomorphic so we can identify the respective basis vectors and use the inner product of $\so(3).$ 

%\bibliographystyle{plain}        % Include this if you use bibtex 
\bibliography{autosam}     

@book{holm2009geometric,
  title={Geometric mechanics and symmetry: from finite to infinite dimensions},
  author={Holm, Darryl D and Schmah, Tanya and Stoica, Cristina},
  volume={12},
  year={2009},
  publisher={Oxford University Press}
}

@article{o1967submersions,
  title={Submersions and geodesics},
  author={O’Neill, Barrett},
  year={1967},
  journal={Duke Mathematical Journal},
  volume={34},
  number={2},
  pages={363-373},
  doi={10.1215/S0012-7094-67-03440-0}
}

@book{o1983semi,
  title={Semi-Riemannian Geometry With Applications to Relativity},
  author={O'Neill, B.},
  isbn={9780080570570},
  series={ISSN},
  year={1983},
  publisher={Elsevier Science}
}

@incollection{bloch2003nonholonomic,
  title={Nonholonomic mechanics},
  author={Bloch, Anthony M},
  booktitle={Nonholonomic mechanics and control},
  pages={},
  year={2003},
  publisher={Springer}
}

@book{ne_mark2004dynamics,
  title={Dynamics of nonholonomic systems},
  author={Neimark, Juru Isaakovich and Fufaev, Nikola Alekseevich},
  volume={33},
  year={2004},
  publisher={American Mathematical Soc.}
}

@article{bloch2017optimal,
  title={Optimal control problems with symmetry breaking cost functions},
  author={Bloch, Anthony M and Colombo, Leonardo J and Gupta, Rohit and Ohsawa, Tomoki},
  journal={SIAM Journal on Applied Algebra and Geometry},
  volume={1},
  number={1},
  pages={626--646},
  year={2017},
  publisher={SIAM}
}

@article{colombo2023lie,
  title={Lie-Poisson reduction for optimal control of left-invariant control systems with subgroup symmetry},
  author={Colombo, Leonardo and Stratoglou, Efstratios},
  journal={Reports on Mathematical Physics},
  volume={91},
  number={1},
  pages={131--141},
  year={2023},
  publisher={Elsevier}
}

@article{stratoglou2023virtual,
  title={Virtual Constraints on Lie groups},
  author={Stratoglou, E and Anahory Simoes, A and Bloch, A and Colombo, L},
  journal={arXiv preprint arXiv:2312.17531},
  year={2023}
}

@article{Simoes:linear:nonholonomic,
title = {Virtual nonholonomic constraints: A geometric approach},
journal = {Automatica},
volume = {155},
pages = {111166},
year = {2023},
issn = {0005-1098},
author = {Alexandre Anahory Simoes and Efstratios Stratoglou and Anthony Bloch and Leonardo J. Colombo}
}

@article{goodman2024b-reduction,
  title={Reduction by Symmetry and Optimal Control with Broken Symmetries on Riemannian Manifolds},
  author={Goodman, Jacob R and Colombo, Leonardo J},
  journal={arXiv preprint arXiv:2401.01129},
  year={2024}
}

@article{goodman2024reduction,
  title={Reduction by symmetry in obstacle avoidance problems on riemannian manifolds},
  author={Goodman, Jacob R and Colombo, Leonardo J},
  journal={SIAM Journal on Applied Algebra and Geometry},
  volume={8},
  number={1},
  pages={26--53},
  year={2024},
  publisher={SIAM}
}

@book{westervelt2018feedback,
  title={Feedback control of dynamic bipedal robot locomotion},
  author={Westervelt, Eric R and Grizzle, Jessy W and Chevallereau, Christine and Choi, Jun Ho and Morris, Benjamin},
  year={2018},
  publisher={CRC press}
}

@article{bloch1995nonholonomic,
  title={Nonholonomic control systems on Riemannian manifolds},
  author={Bloch, Anthony M and Crouch, Peter E},
  journal={SIAM Journal on Control and Optimization},
  volume={33},
  number={1},
  pages={126--148},
  year={1995},
  publisher={SIAM}
}

@book{isidori1985nonlinear,
  title={Nonlinear control systems: an introduction},
  author={Isidori, Alberto},
  year={1985},
  publisher={Springer}

}

@book{jurdjevic1997geometric,
  title={Geometric control theory},
  author={Jurdjevic, Velimir},
  year={1997},
  publisher={Cambridge university press}
}

@article{moran2023gymnastics,
  title={From Gymnastics to Virtual Nonholonomic Constraints: Energy Injection, Dissipation, and Regulation for the Acrobot},
  author={Moran-MacDonald, Adan and Maggiore, Manfredi and Wang, Xingbo},
  journal={IEEE Transactions on Control Systems Technology},
  year={2023}
}

@article{stratoglou2024virtualc,
  title={Virtual constraints on {R}iemannian homogeneous spaces},
  author={Stratoglou, Efstratios and Simoes, Alexandre Anahoy and Bloch, Anthony and Colombo, Leonardo},
  journal={IFAC-PapersOnLine},
  volume={58},
  number={6},
  pages={77--82},
  year={2024},
  publisher={Elsevier}
}

@inproceedings{stratoglou2023bvirtual,
  title={Virtual Affine Nonholonomic Constraints},
  author={Stratoglou, Efstratios and Simoes, Alexandre Anahory and Bloch, Anthony and Colombo, Leonardo},
  booktitle={International Conference on Geometric Science of Information},
  pages={89--96},
  year={2023},
  organization={Springer}
}

@article{stratoglou2023geometry,
  title={On the Geometry of Virtual Nonlinear Nonholonomic Constraints},
  author={Stratoglou, Efstratios and Simoes, Alexandre Anahory and Bloch, Anthony and Colombo, Leonardo J},
  journal={arXiv preprint arXiv:2310.01849},
  year={2023}
}

@book{moran2021energy,
  title={Energy injection for mechanical systems through the method of Virtual Nonholonomic Constraints},
  author={Moran-MacDonald, Adan},
  year={2021},
  publisher={University of Toronto (Canada)}
}

@inproceedings{griffin2015nonholonomic,
  title={Nonholonomic virtual constraints for dynamic walking},
  author={Griffin, Brent and Grizzle, Jessy},
  booktitle={2015 54th IEEE Conference on Decision and Control (CDC)},
  pages={4053--4060},
  year={2015},
  organization={IEEE}
}

@article{westervelt2003hybrid,
  title={Hybrid zero dynamics of planar biped walkers},
  author={Westervelt, Eric R and Grizzle, Jessy W and Koditschek, Daniel E},
  journal={IEEE transactions on automatic control},
  volume={48},
  number={1},
  pages={42--56},
  year={2003},
  publisher={IEEE}
}

@article{chevallereau2003rabbit,
  title={Rabbit: A testbed for advanced control theory},
  author={Chevallereau, Christine and Abba, Gabriel and Aoustin, Yannick and Plestan, Franck and Westervelt, Eric and De Wit, Carlos Canudas and Grizzle, Jessy},
  journal={IEEE Control Systems Magazine},
  volume={23},
  number={5},
  pages={57--79},
  year={2003}
}

@article{freidovich2008periodic,
  title={Periodic motions of the pendubot via virtual holonomic constraints: Theory and experiments},
  author={Freidovich, Leonid and Robertsson, Anders and Shiriaev, Anton and Johansson, Rolf},
  journal={Automatica},
  volume={44},
  number={3},
  pages={785--791},
  year={2008},
  publisher={Elsevier}
}

@article{mohammadi2018dynamic,
  title={Dynamic virtual holonomic constraints for stabilization of closed orbits in underactuated mechanical systems},
  author={Mohammadi, Alireza and Maggiore, Manfredi and Consolini, Luca},
  journal={Automatica},
  volume={94},
  pages={112--124},
  year={2018},
  publisher={Elsevier}
}

@book{helgason1979differential,
  title={Differential geometry, Lie groups, and symmetric spaces},
  author={Helgason, Sigurdur},
  year={1979},
  publisher={Academic press}
}

@article{canudas2004concept,
  title={On the concept of virtual constraints as a tool for walking robot control and balancing},
  author={Canudas-de-Wit, Carlos},
  journal={Annual Reviews in Control},
  volume={28},
  number={2},
  pages={157--166},
  year={2004},
  publisher={Elsevier}
}

@book{bullo2019geometric,
  title={Geometric control of mechanical systems: modeling, analysis, and design for simple mechanical control systems},
  author={Bullo, Francesco and Lewis, Andrew D},
  volume={49},
  year={2019},
  publisher={Springer}
}

@book{boothby1986introduction,
  title={An introduction to differentiable manifolds and Riemannian geometry},
  author={Boothby, William M},
  year={1986},
  publisher={Academic press}
}

@article{Consol:Costal:Maggiore,
    author = {Luca Consolini and Alessandro Costalunga and Manfredi Maggiore} ,
    title = {A coordinate-free theory of virtual holonomic constraints},
    journal = {Journal of Geometric Mechanics},
    year = {2018},
    volume = {10},
    number = {4},
    pages = {467-502},
    issn = {1941-4889},
    doi = {10.3934/jgm.2018018},
}

@article{Consol:Constal,
  author={Consolini, Luca and Costalunga, Alessandro},
  booktitle={2015 54th IEEE Conference on Decision and Control (CDC)}, 
  title={Induced connections on virtual holonomic constraints}, 
  year={2015},
  volume={},
  number={},
  pages={139-144},
  doi={10.1109/CDC.2015.7402099}
}

@article{rojo2010rolling,
  title={The rolling sphere, the quantum spin, and a simple view of the Landau--Zener problem},
  author={Rojo, Alberto G and Bloch, Anthony M},
  journal={American Journal of Physics},
  volume={78},
  number={10},
  pages={1014--1022},
  year={2010},
  publisher={AIP Publishing}
}

\end{document}